\documentclass[11pt,letterpaper]{article}
\usepackage{amsmath,amssymb,amsthm}
\usepackage{MnSymbol} 
\usepackage{graphicx,color}
\usepackage[hyphens]{url}
\usepackage{dsfont}
\usepackage{booktabs}
\usepackage[normalem]{ulem}
\usepackage{mathtools}
\usepackage[hidelinks,colorlinks=true,linkcolor=blue,citecolor=blue]{hyperref}
\usepackage[nameinlink,capitalize]{cleveref}
\usepackage{multirow}
\usepackage{algorithm}
\usepackage[noend]{algpseudocode}
\usepackage{tikz} %

\usepackage{enumitem}

\usepackage{soul} %
\usepackage[square,sort,numbers]{natbib}
\usepackage[sort,nocompress,noadjust]{cite}

\usepackage{subcaption}

\usepackage{tabularx}
\newcolumntype{L}{>{\raggedright\arraybackslash}X}

\numberwithin{equation}{section}
\newtheorem{theorem}{Theorem}[section]
\newtheorem{cor}[theorem]{Corollary}
\newtheorem{lemma}[theorem]{Lemma}
\newtheorem{remark}[theorem]{Remark}

\newtheorem{defin}[theorem]{Definition}

\newcommand{\cF}{\mathcal{F}}

\newcommand{\cI}{\mathcal{I}}

\newcommand{\cP}{\mathcal{P}}

\newcommand{\R}{\mathbb{R}}

\newcommand{\N}{\mathbb{N}}

\DeclareMathOperator{\logeps}{\log\tfrac{1}{\eps}}

\newcommand*{\eps}{\varepsilon}

\newcommand*{\lam}{\lambda}
\newcommand*{\sig}{\sigma}
\newcommand*{\kap}{\kappa}

\newcommand{\plusminus}{\raisebox{.2ex}{$\scriptstyle\pm$}}
\DeclareMathOperator*{\argmin}{argmin}

\renewcommand{\leq}{\leqslant}
\renewcommand{\geq}{\geqslant}

\usepackage[margin=1in]{geometry}

\makeatletter
\def\blfootnote{\gdef\@thefnmark{}\@footnotetext}
\makeatother

\begin{document}

\title{Acceleration by Stepsize Hedging I: \\ Multi-Step Descent and the Silver Stepsize Schedule}
	
	\author{
		Jason M. Altschuler
		\\	UPenn \\	\texttt{alts@upenn.edu}
		\and
		Pablo A. Parrilo \\
		LIDS - MIT \\	\texttt{parrilo@mit.edu}
	}
	\date{\today}
	\maketitle

\begin{abstract}	
   Can we accelerate convergence of gradient descent without changing the algorithm---just by carefully choosing stepsizes? Surprisingly, we show that the answer is yes. 
   Our proposed \emph{Silver Stepsize Schedule} optimizes strongly convex functions in $\kappa^{\log_{\rho} 2} \approx \kappa^{0.7864}$ iterations, where $\rho=1+\sqrt{2}$ is the silver ratio and $\kappa$ is the condition number.
	 This is intermediate between the textbook unaccelerated rate $\kappa$ and the accelerated rate $\kappa^{1/2}$ due to Nesterov in 1983. The non-strongly convex setting is conceptually identical, and standard black-box reductions imply an analogous partially accelerated rate $\eps^{-\log_{\rho} 2} \approx \eps^{-0.7864}$. We conjecture and provide partial evidence that these rates are optimal among all stepsize schedules. 

    \par The Silver Stepsize Schedule is constructed recursively in a fully explicit way. It is
    non-monotonic, fractal-like, and approximately periodic of period $\kappa^{\log_{\rho} 2}$. This leads to a phase transition in the convergence rate: initially super-exponential (acceleration regime), then exponential (saturation regime).

    \par The core algorithmic intuition is \emph{hedging} between individually suboptimal strategies---short steps and long steps---since bad cases for the former are good cases for the latter, and vice versa. Properly combining these stepsizes yields faster convergence due to the misalignment of worst-case functions. The key challenge in proving this speedup is enforcing long-range consistency conditions along the algorithm's trajectory. We do this by developing a technique that recursively glues constraints from different portions of the trajectory, thus removing a key stumbling block in previous analyses of optimization algorithms. More broadly, we believe that the concepts of hedging and multi-step descent have the potential to be powerful algorithmic paradigms in a variety of contexts in optimization and beyond.

	\par This series of papers publishes and extends the first author’s 2018 Master’s Thesis (advised by the second author)---which established for the first time that carefully choosing stepsizes can enable acceleration in convex optimization.  Prior to this thesis, the only such result was for the special case of quadratic optimization, due to Young in 1953.
\end{abstract}

	\newpage
	\setcounter{tocdepth}{2}
	\tableofcontents
	
	\newpage
	
\section{Introduction}\label{sec:intro}

Gradient descent (GD) is a simple iterative algorithm to minimize an objective function $f$ by producing better and better estimates via the update
\begin{align}
	x_{t+1} = x_t  - \alpha_t \nabla f(x_t)\,, \qquad \forall t = 0, 1, 2, \dots.
	\label{eq:gd}
\end{align}
GD dates back nearly two hundred years to the work of Cauchy~\citep{cauchy1847}, yet it (and its variants) remain a primary workhorse in modern optimization, engineering, and machine learning due to the practical efficacy, 
simplicity, and scalability. 
It is of both theoretical and practical importance to analyze the convergence of GD and moreover to optimize parameters so that this convergence is as fast as possible.

\par A central fact in convex optimization is that with a prudent choice of the stepsize schedule $\{\alpha_t\}$---the only\footnote{In convex optimization, we typically view the initialization $x_0$ as part of the problem instance rather than a parameter choice, since $x_0 = 0$ without loss of generality after a possible translation of the objective function $f$.} parameters of the algorithm---running GD from any initialization $x_0$ produces iterates which optimize $f$ to arbitrary accuracy. 
Quantifying this statement leads to two intertwined questions: How fast does $x_n$ converge to a minimizer $x^*$ of $f$? And what stepsize choice $\{\alpha_t\}$ leads to the fastest convergence rate?
\par This series of papers revisits these classical questions in the fundamental setting of smooth\footnote{In the non-smooth setting, it is classically known that acceleration is impossible, and moreover GD achieves the minimax-optimal convergence rate  with simple monotonically decaying stepsize schedules like $\alpha_t \asymp 1/\sqrt{t}$~\citep{nesterov-survey}.} convex optimization. Our overarching goal is to understand how much mileage can be obtained by simply optimizing the stepsize choice for GD. 

\par Note that this is markedly different from the past forty years of literature on accelerating the convergence rate for GD. That literature---starting from Nesterov's seminal work in 1983~\citep{nesterov-agd}---achieves faster convergence rates by modifying the GD algorithm with extra building blocks such as momentum, auxiliary sequences, or other internal dynamics. See the related work section or the recent survey~\citep{d2021acceleration}. In contrast, we investigate the basic question of: can we accelerate convergence without changing the GD algorithm---just by optimizing the stepsizes?

\paragraph*{Mainstream approach.} The standard analysis of GD uses a constant stepsize schedule, i.e., $\alpha_t = \bar{\alpha}$ for all iterations $t$; see e.g.~the textbooks~\citep{bubeck-book,boyd,polyakbook,hazan2016introduction,nesterov-survey,luenberger1984linear,BertsekasNonlinear} among many others. For example, 
$\bar{\alpha} = 1/M$ in the setting of $M$-smooth convex objectives, or $\bar{\alpha} = 2/(M+m)$ if the objectives are additionally $m$-strongly convex. This prescription is based on the following fact: 
\begin{gather}
	\text{For one iteration of GD, there is a unique stepsize } \bar{\alpha} \text{ achieving the fastest convergence rate} \nonumber
	\\
	\text{(in the worst case over functions and initializations).}
	\label{eq:intro:1-step}
\end{gather}
This is provably correct. For example, in the strongly convex setting, this $\bar{\alpha}$ provides the optimal contraction rate---a larger stepsize $\alpha_t > \bar{\alpha}$ can lead to overshooting the target $x^*$, and a smaller stepsize $\alpha_t < \bar{\alpha}$ can lead to undershooting $x^*$.

\par However, it is well-known that even after optimizing the constant $\bar{\alpha}$, this constant stepsize schedule leads to a slow convergence rate. (Hence the intensive research on accelerated GD.) Moreover, even though many alternative stepsize schedules have been proposed in both theory and practice---e.g.,  exact line search, Armijo-Goldstein rules, Polyak-type schedules, Barzilai-Borwein-type schedules, etc., see the related work section---none of these alternative schedules have led to an analysis that outperforms the slow ``unaccelerated'' rate of constant stepsize GD. Conventional wisdom therefore dictates that slow convergence is unavoidable, unless one modifies GD by adding extra building blocks beyond choosing stepsizes, e.g., via momentum.

\begin{table}[]
	\small
	\begin{tabular}{|c|c|c|}
		\hline
		& \textbf{Quadratic }                          & \textbf{Convex      }                              \\ \hline
		\textbf{Mainstream stepsizes}  & $\Theta(\kappa)$ by constant stepsizes (folklore)  & $\Theta(\kappa)$ by constant stepsizes (folklore)   \\
		\textbf{Additional dynamics} & $\Theta(\sqrt{\kappa})$ by Heavy Ball~\citep{polyak1964some}          & $\Theta(\sqrt{\kappa})$ by Nesterov Acceleration~\citep{nesterov-agd}     \\
	\textbf{Hedged stepsizes}    & $\Theta(\sqrt{\kappa})$ by Chebyshev Stepsizes~\citep{young53} & \color{blue}{ $\Theta(
		\kappa^{\log_{\rho} 2}
		)$ by Silver Stepsizes (Theorem~\ref{thm:main})} \\ \hline
	\end{tabular}
	\caption{\footnotesize{Iteration complexity of various approaches for minimizing a $\kappa$-conditioned function. The dependence on the accuracy $\eps$ is omitted as it is always $\log 1/\eps$. Mainstream stepsize schedules require $\Theta(\kappa)$ iterations; this is the textbook unaccelerated rate. For the special case of quadratics (left), accelerated rates of $\Theta(\sqrt{\kappa})$ can be equivalently achieved via Young's 1953 Chebyshev Stepsize Schedule~\citep{young53} or Polyak's 1964 Heavy Ball Algorithm~\citep{polyak1964some}. For the general case of convex functions (right), this equivalence between internal dynamics and varying stepsizes is false. Acceleration was first achieved by Nesterov's 1983 Fast Gradient Algorithm~\citep{nesterov-agd} and it has long been believed that in the convex setting, any acceleration requires modifying GD by adding internal dynamics, e.g., momentum. We prove that accelerated convex optimization \emph{is} possible by choosing better stepsizes.
		}}
\end{table}

\paragraph*{Faster convergence via dynamic stepsizes?} The premise of this series of papers
is that this is wrong. Why might the constant stepsize schedule $\alpha_t = \bar{\alpha}$ be sub-optimal?  Certainly it is optimal if GD is only run for $n=1$ iteration---this is the assertion~\eqref{eq:intro:1-step}. However, it is sub-optimal for $n$ steps of GD, for any $n > 1$. Briefly, this is because the statement for $n=1$ requires the worst-case problem instance (the objective function $f$ and initialization $x_0$) to align with the choice of stepsize $\alpha_t \neq \bar{\alpha}$ so that the convergence is slow, and for $n > 1$, the worst-case problem instances for each individual step might not align. This suggests an algorithmic opportunity: 
\begin{gather}
	\text{Is it possible to combine (individually suboptimal) stepsizes } \alpha_t \neq \bar{\alpha} \nonumber \\
	\text{to achieve faster convergence for } n > 1\text{ iterations?}
	\label{eq:intro:hedge}
\end{gather}
We refer to this algorithmic idea as \emph{hedging} between worst-case problem instances. (See \S\ref{sec:hedging} for a fully worked-out example.)

\paragraph*{Motivation: the special case of quadratics.} Of course, using non-constant stepsizes is not a new idea---for the special case of minimizing \emph{convex quadratics}, it has been known that this enables faster convergence since Young's seminal paper in 1953~\citep{young53}. In particular, for quadratic optimization, the optimal stepsize schedule is not constant, but given by the inverse roots of Chebyshev polynomials; the order of these stepsizes is irrelevant for the convergence rate (assuming exact arithmetic); and the resulting convergence rate is the so-called \emph{accelerated rate} that is optimal among all Krylov-subspace algorithms~\citep{nem-yudin}, including even modifications of GD that use momentum or internal dynamics. See the related work section \S\ref{ssec:intro:related} for further details.

\paragraph*{A longstanding gap between quadratic and convex optimization.} However, while the advantage of non-constant stepsizes has been well-understood for quadratic optimization for 70 years (and nowadays is even taught in many introductory optimization courses), it has remained entirely open whether this phenomenon extends to any setting of convex optimization beyond quadratics. In particular, it was unclear whether \emph{any} stepsize schedule could lead to \emph{any} speedup over the textbook GD convergence rate---even by a constant factor.
\par This gap is due to several reasons. First, many phenomena from the quadratic case are simply false in the setting of general convex optimization: e.g., the stepsize schedule based on roots of the Chebyshev polynomials is provably bad for the convex setting~\citep[Chapter 8]{altschuler2018greed}, and the order of the stepsizes dramatically affects the convergence rate in the convex setting~\citep[Chapter 8]{altschuler2018greed}. Second, any approach for establishing the advantage of a non-constant stepsize schedule must track how progress in the current iteration is affected by previous iterations---and this effect of history appeared to only be explicitly computable in the quadratic setting, essentially since that is the only case in which the GD map is linear (hence tractable to track after repeated iterations).

\subsection{Contribution and discussion}\label{ssec:intro:cont}

In this initial paper, we show that GD can converge faster for smooth convex optimization by using certain time-varying, non-monotonic stepsize schedules. This answers the hedging question~\eqref{eq:intro:hedge} in the affirmative. This series of papers publishes and extends the first author's 2018 Master's Thesis~\citep{altschuler2018greed} (advised by the second author), which proved such a result for the first time, see the related work section \S\ref{ssec:intro:related}. In particular, Chapter 8 of the thesis showed for the first time that a constant-factor improvement over the unaccelerated rate is possible in the smooth strongly convex setting, and Chapter 6 of the thesis showed for the first time that an asymptotic acceleration is possible in any setting beyond quadratics. (The latter result proves that arcsine-distributed random stepsizes achieve the fully accelerated rate $\Theta(\sqrt{\kappa} \log1/\eps)$ if the convex functions are separable; this will be detailed in a forthcoming paper.) Prior to this thesis, the only result for acceleration via choosing stepsizes was for the special case of quadratic optimization, due to Young in 1953.

Conceptually, we deviate from traditional analyses of GD (and other optimization algorithms) by directly analyzing the cumulative progress of all the steps of the algorithm, rather than combining separate bounds for the progress of individual steps. As mentioned above, this global analysis of \emph{multi-step descent} is provably necessary to show any benefit for any deviation from the constant stepsize schedule. Indeed, separately analyzing the progress for each iteration---as done, e.g., in standard GD analyses, in exact line search, or in standard offline-to-online convex optimization reductions---is provably too shortsighted and unavoidably leads to pessimistic, unaccelerated convergence rates. The key difficulty is how to track how different iterations affect progress in other iterations. Previously, this could be accomplished only for the special case of quadratics because then the GD update is linear. We show that this can be accomplished for general convex setting by this by using long-range consistency conditions between the gradients seen along the algorithm's trajectory. We provide a high-level overview of these new conceptual ideas in \S\ref{sec:hedging}.

Below, we formally state our main result in \S\ref{sssec:intro:dis:main}, and then discuss the improved convergence rate in \S\ref{sssec:intro:dis:rate}, the proposed stepsize schedule in \S\ref{sssec:intro:dis:schedule}, and the generality of the result in \S\ref{sssec:intro:dis:setting}.

\subsubsection{Main result: acceleration without momentum}\label{sssec:intro:dis:main}

\begin{figure}[t]
	\centering
 	\includegraphics[width=.24\columnwidth]{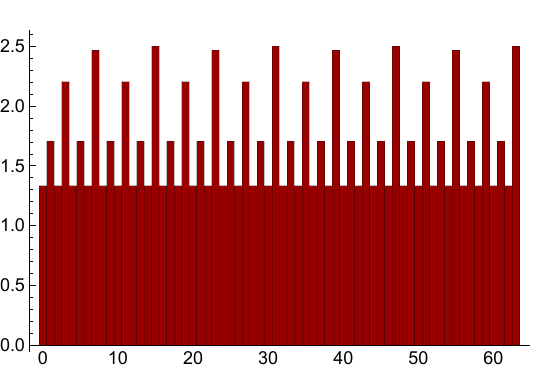}
   	\includegraphics[width=.24\columnwidth]{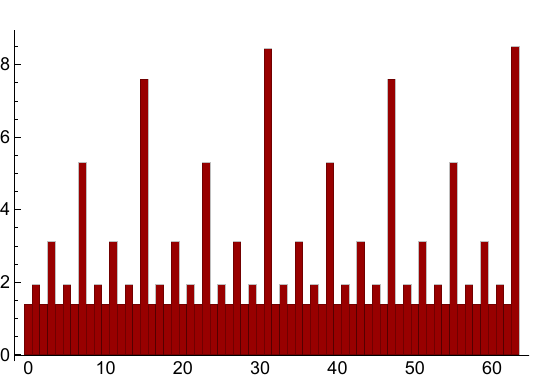}
 	\includegraphics[width=.24\columnwidth]{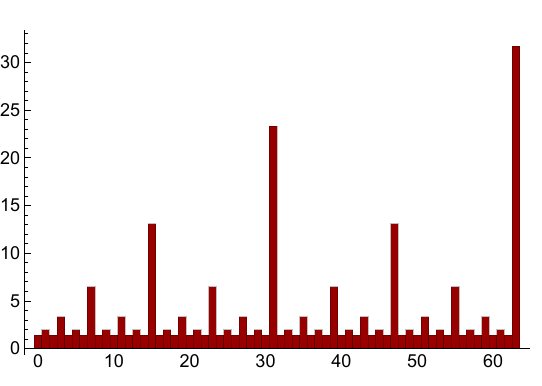}
 	\includegraphics[width=.24\columnwidth]{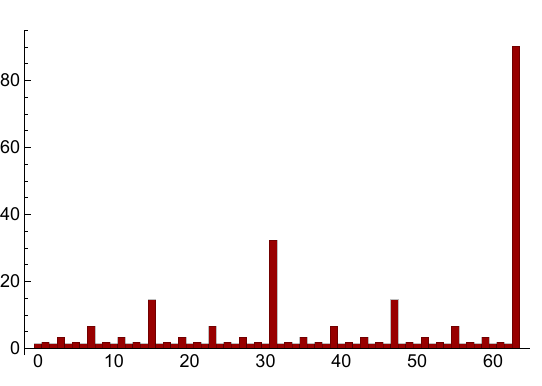}

	\caption{\footnotesize Silver Stepsize Schedule, for different condition numbers $\kappa = 4,16,64,256$ -- only the first 64 stepsizes are shown. Notice the recursive, fractal behavior and the approximate periodicity with period of size $n^* = \kappa^{\log_{\rho} 2}$; details in \S\ref{sssec:intro:dis:schedule}. Also note the different scales on the vertical axis, since the stepsizes are unnormalized, vs. the normalized stepsizes in Figure~\ref{fig:normalizedstepsizes}.}
	\label{fig:stepsizes}
\end{figure}

Formalizing this result requires restricting to a function class with controlled curvature. For concreteness, in this first paper we focus on the well-studied setting of strongly convex and smooth $f$, and we measure progress via distance to the optimum $x^*$. While smoothness is classically known to be required for acceleration~\citep{nesterov-survey}, the other choices and assumptions in the theorem statement are not essential: strong convexity can be relaxed to convexity, and the progress measure can be replaced with other standard desiderata; see the discussion in \S\ref{sssec:intro:dis:setting}. Below, let $\rho := 1 + \sqrt{2}$ denote the silver ratio, and assume throughout that $f$ is $\kappa$-conditioned, i.e., $1$-strongly convex and $\kappa$-smooth\footnote{Recall that this means $f$ is sandwiched between quadratic lower and upper bounds of curvature $1$ and $\kappa$, respectively, i.e., $ \frac{1}{2}\|v\|^2 \leq f(x+v) - f(x) - \langle v, \nabla f(x) \rangle \leq \frac{\kappa}{2}\|v\|^2$ for all $x, v \in \R^d$. For intuition, this is equivalent to the local curvature bound $I_d \preceq \nabla^2 f(x) \preceq \kappa I_d$ under the assumption of twice-differentiability (not required by our results).}---this is without loss of generality after rescaling.

\begin{theorem}\label{thm:main}
	For any horizon $n \in \N$ that is a power of $2$, any dimension $d$, any $\kappa$-conditioned function $f : \R^d \to \R$, and any initialization $x_0$,
	\begin{align}
		\|x_n - x^*\|^2 \leq \tau_n \|x_0 - x^*\|^2\,,
		\label{eq:thm-main:cert}
	\end{align}
	where $x^*$ denotes the unique minimizer of $f$, $x_n$ denotes the output of $n$ steps of GD using the Silver Stepsize Schedule (defined in \S\ref{sec:construction}), and $\tau_n$ denotes the $n$-step Silver Convergence Rate (defined in \S\ref{sec:construction}). Moreover, $\tau_n$ undergoes the following phase transition at $n^* = \Theta( \kappa^{\log_{\rho} 2} )$:
	\begin{itemize}
		\item \underline{Acceleration regime.} For $n \leq n^*$,
		\[
		\tau_n = \exp\left( - \Theta\left( \frac{n^{\log_2 \rho}}{\kappa}  \right) \right)\,.
		\]
		\item \underline{Saturation regime.} For $n > n^*$, 
		\[
		\tau_n = \exp\left( - \Theta\left( \frac{n}{n^*}  \right) \right)\,.
		\]
	\end{itemize}
	In particular, in order to achieve a final error $\|x_n - x^*\|^2 \leq \eps$, it suffices to run GD using the Silver Stepsize Schedule for 
	\begin{align}
		n = \Theta\left(  \kappa^{\log_{\rho} 2} \logeps \right) \approx \Theta\left( \kappa^{0.7864} \logeps\right) \;\; \textrm{iterations}\,.
	\end{align}
\end{theorem}

\begin{figure}
    \centering
    \includegraphics[height=6cm]{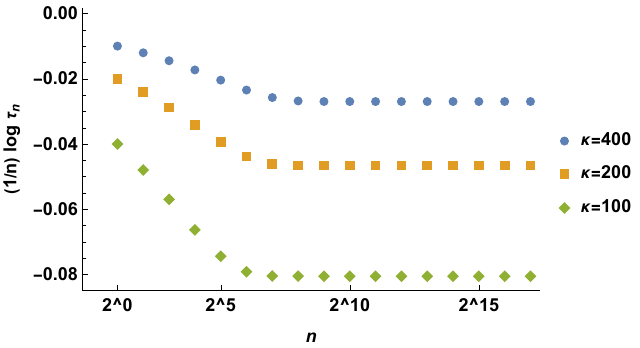}
    \caption{\footnotesize Log of the average per-step rate, aka $\tfrac{1}{n} \log \tau_n$,
    for varying condition numbers $\kappa$. The initial value is the unaccelerated rate 
    $(\tfrac{\kappa -1}{\kappa + 1})^2$.
    Notice the rate saturation phenomenon that occurs at $n=n^* \asymp \kappa^{\log_{\rho} 2}$. 
    	}
    \label{fig:ratesaturation}
\end{figure}

\subsubsection{Discussion of Silver Convergence Rate}\label{sssec:intro:dis:rate}

\paragraph*{Partial acceleration.} Our rate $\Theta(\kappa^{\log_{\rho} 2} \log 1/\eps)$ lies between the textbook rate $\Theta(\kappa \log1/\eps)$ for GD and the accelerated rate $\Theta(\sqrt{\kappa} \log1/\eps)$ due to Nesterov in 1983~\citep{nesterov-agd}. We emphasize that before the thesis~\citep{altschuler2018greed} that this paper is based upon, it was unknown if any improvement over the unaccelerated rate---even a constant factor---was achievable by any stepsize schedule. Our convergence rate is faster than all known GD stepsize schedules for convex optimization, including constant stepsize schedules, Polyak-type schedules, Barzilai-Borwein-type schedules, Goldstein-Armijo-type schedules, exact line search, etc.

\paragraph*{Phase transition.} 
A distinctive feature of the Silver Convergence Rate $\tau_n$ is that it undergoes a phase transition: $\tau_n$ switches from super-exponential to exponential in the horizon $n$. This transition occurs at $n^* \asymp \kappa^{\log_{\rho} 2}$, which is the number of iterations required to make the error decrease by a constant factor. See Figure~\ref{fig:ratesaturation}. The reason for these two regimes is that beyond $n^*$, the new stepsizes converge quadratically fast to their stationary value; details in \S\ref{sec:rate}. 
\begin{itemize}
	\item \underline{Acceleration regime.} This regime encapsulates the advantage of multi-step descent: the super-exponentiality of the $n$-step bound makes it better than composing the $1$-step bound $n$ times. This super-exponential regime interpolates the $\kappa$ dependence between the unaccelerated rate (achieved at $n = 1$) and our partially accelerated rate (achieved at $n \gtrsim n^*$).
	\item \underline{Saturation regime.} Here, the benefit of multi-step descent becomes negligible: $\tau_{2n} \approx \tau_n^2$ for $n \geq n^*$.\footnote{Note that $\tau_{2n} \leq \tau_n^2$; intuitively this amounts to the statement that the optimal $2n$-step schedule is at least as good as repeating the optimal $n$-step schedule twice. We call this inequality \emph{rate monotonicity}, see \S\ref{sec:rate}. The statement $\tau_{2n} \approx \tau_n^2$ therefore states that this bound is nearly tight.} Briefly, this rate saturation occurs because the Silver Stepsize Schedule is approximately periodic with period $n^*$, see \S\ref{sssec:intro:dis:schedule}. 
\end{itemize}

\paragraph*{Dimension independence.} The convergence rate in Theorem~\ref{thm:main} is independent of the dimension $d$ and thus can be extended to infinite-dimensional Hilbert space. This is because our analysis only uses consistency conditions for the GD trajectory to arise from a convex function---and these consistency conditions are dimension-independent~\citep{rockafellar,pesto}. This is in common with classical analyses of GD and Nesterov-style acceleration.

\paragraph*{Optimality.} We conjecture the Silver Stepsize Schedule has the fastest convergence rate among all possible choices of GD stepsize schedules. We prove optimality for the $n=2$ case of~\citep[Chapter 8]{altschuler2018greed} in \S\ref{sec:hedging}; this proof readily extends to small $n$, and we will address the question of optimality for all $n$ in a shortly forthcoming paper.

\subsubsection{Discussion of Silver Stepsize Schedule}\label{sssec:intro:dis:schedule}

\paragraph*{Recursive construction.} The Silver Stepsize Schedule is defined recursively in a fully explicit way. We briefly overview the construction; see \S\ref{sec:construction} for full details.
The $1$-step schedule $h^{(1)}$ is initialized to the constant $\bar{\alpha} = 2/(1 + 1/\kappa)$ that is classically known to be optimal for $1$-step descent. We then recursively define the $2n$-step schedule $h^{(2n)}$ as
\begin{align}
	h^{(2n)} := [ \tilde{h}^{(n)}, a_{2n}, \tilde{h}^{(n)}, b_{2n}]\,,
\end{align}
where $\tilde{h}^{(n)}$ is the $n$-step schedule $h^{(n)}$ with its final stepsize $b_n$ removed, and $a_{2n}$ and $b_{2n}$ are obtained by ``splitting'' this removed stepsize $b_n$. Modulo a certain normalizing transformation, this splitting produces $a_{2n} < b_n < b_{2n}$ as the roots to a certain quadratic equation in $b_n$. See \S\ref{sec:rate} for details and closed-form expressions.

\paragraph*{Finite-horizon schedule.} This recursive construction produces (normalized) stepsize schedules that follow the pattern
\begin{align*}
	h^{(1)} &= [b_1] \\
	h^{(2)} &= [a_2,b_2] \\
	h^{(4)} &= [a_2,a_4,a_2,b_4] \\
	h^{(8)} &= [a_2,a_4,a_2,a_8,a_2,a_4,a_2,b_8]
\end{align*}
See Figure~\ref{fig:stepsizes} for a visualization. 

\paragraph*{Infinite-horizon schedule.} This schedule simplifies in the limit $n \to \infty$: the $i$-th normalized stepsize is given by $a_{B(i)}$, where $B(i)$ denotes the smallest power of $2$ in the binary expansion of $i$. Note that no entries of the $b$ sequence appear.

\paragraph*{Fractal order.} For the special case of quadratic optimization, the order of the stepsizes is well-known to be irrelevant for the convergence rate. In contrast, in the general setting of convex optimization, the order of the stepsizes provably does matter~\citep[Chapter 8]{altschuler2018greed}. For example, it can be shown that the convergence rate in Theorem~\ref{thm:main} becomes greater than $1$ (i.e., not even contractive) if one reverses the order of the $2$-step Silver Stepsize Schedule. 
\par The Silver Stepsize Schedule generates a fractal, see Figure~\ref{fig:stepsizes}. This is due to our recursive construction, and is directly evident from the aforementioned fact that the $i$-th stepsize depends on the sparsity pattern of the binary expansion of $i$. This fractal structure aligns with the numerical observations in~\citep{gupta22,grimmer23}, and is in stark contrast with all classical stepsize schedules which, if time-varying, decay monotonically in the iteration number $i$, e.g., as $1/i$.

\paragraph*{Approximate periodicity.} The Silver Stepsize Schedule is not periodic as it is continually changes. However, it is approximately periodic with period $n^* \asymp \kappa^{\log_{\rho} 2}$, see Figure~\ref{fig:stepsizes}. This is another facet of the rate saturation phenomenon discussed in \S\ref{sssec:intro:dis:rate}. See \S\ref{sec:rate} for details.

\paragraph*{Dependence on horizon.} Theorem~\ref{thm:main} is stated for horizons $n$ that are powers of $2$. For arbitrary integers $n$, one can simply run the Silver Stepsize Schedule for the largest power of $2$ below $n$, or better, run for all powers of $2$ in the binary expansion of $n$. This affects the average per-step-rate by only a small constant factor. We moreover conjecture that simply using $n$ steps of the infinite-horizon Silver Stepsize Schedule leads to the same convergence rate modulo a lower-order term. This seems reasonable since only logarithmically many stepsizes are changed, but we have not attempted to prove this. Orthogonally, if the horizon is not set in advance, then one can, e.g., do a ``doubling'' trick by exploiting the fact that the first $2^{i}-1$ stepsizes are identical for all $n \geq 2^i$. Specifically, for each $i$, decide on iteration $2^{i} - 1$ whether to stop at $n=2^i$ iterations, or repeat roughly the same amount of effort and go to $2^{i+1}-1$ iterations.

\subsubsection{Discussion of problem setting}\label{sssec:intro:dis:setting}

\paragraph*{Progress measure.} Theorem~\ref{thm:rate} uses distance as the progress measure. This can be replaced by other standard progress measures such as function suboptimality or gradient norm, in the initial or final condition or both, since these measures are equivalent for $\kappa$-conditioned functions. This black-box replacement affects the rate by only a lower-order term. Moreover, this equivalence factor can be avoided by re-doing our analysis in a conceptually identical way for the desired progress measures (possibly also with minor changes to the stepsize schedule; e.g., for gradient norm contraction, it appears that one should reverse the order~\citep[Chapter 8]{altschuler2018greed}).

\paragraph*{Smoothness.} It is well-known that smoothness is required for acceleration: otherwise, GD cannot be accelerated even with momentum or other internal dynamics~\citep[Chapter 3]{nesterov-survey}.

\paragraph*{Convexity.} Theorem~\ref{thm:main} is stated for the strongly convex setting, but this can be relaxed to the non-strongly convex setting. Indeed, all our core conceptual ideas extend: the advantage of time-varying, non-monotonic stepsizes, proving this advantage via multi-step descent rather than iterating the greedy $1$-step bound, certifying multi-step descent via recursive gluing, etc. The adaptation requires only minor technical modifications to the stepsize schedule, certificate recursion, and progress measure. These details will appear in a shortly forthcoming paper.

We mention that by standard black-box reductions (see e.g.,~\citep{allen2016optimal} or~\citep[page 285]{bubeck-book}), Theorem~\ref{thm:main} immediately implies accelerated rates for the (non-strongly) convex setting by running GD with the Silver Stepsize Schedule on a quadratically regularized objective, i.e., $f(\cdot) + \delta\| \cdot- y\|^2$ for appropriate choices of $\delta$ and $y$. This gives an analogous partially accelerated rate of $\eps^{-\log \rho_2} \approx \eps^{-0.7864}$ iterations to obtain $\eps$ function suboptimality. This is intermediate between the textbook unaccelerated rate $\Theta(\eps^{-1})$ and Nesterov's accelerated rate $\Theta(\eps^{-1/2})$ from 1983~\citep{nesterov-agd}. This strongly suggests that acceleration in the (non-strongly) convex case surpasses the $\Theta(1/(T \log T))$ conjecture in~\citep{grimmer23}. The aforementioned forthcoming paper will address this via a direct analysis that bypasses regularization.

\subsection{Related work}\label{ssec:intro:related}

\subsubsection{The special case of quadratic optimization}\label{sssec:intro:quadratic}

\paragraph*{Three equivalent approaches to acceleration.} For quadratic optimization, the GD map becomes linear, which enables three equivalent approaches to acceleration. One approach, taken by Young in 1953 is to choose non-constant stepsizes that are the inverses of the roots of Chebyshev polynomials~\citep{young53}. A second approach is to use momentum, achieved for example by Hestenes and Stiefel's Conjugate Gradient Method in 1952~\citep{hestenes1952methods} and Polyak's Heavy Ball Method in 1964~\citep{polyak1964some}. This equivalence arises because momentum amounts to a three-term recurrence, which if the coefficients are chosen appropriately, generates the same sequence of Chebyshev polynomials; see e.g.~\citep[Ch. 5]{Varga}. A third approach is to use the limiting distribution of the roots of the Chebyshev polynomials: the arcsine distribution~\citep{kalousek, pronzato11, pronzato13,altschuler2018greed}. This equivalence is due to the fact that the order of stepsizes does not affect convergence in the quadratic case, thus as the horizon $n \to \infty$, one might as well draw stepsizes i.i.d.\,from the equilibrium measure. It is important to emphasize that the elegant equivalences between these three approaches---varying stepsizes, momentum, and equilibrium measures---breaks down beyond the special case of quadratic optimization.

\paragraph*{Desiderata beyond fast convergence.} The above discussion concerns only the convergence rate, not stability. In settings with noisy gradients or inexact arithmetic, the order of the stepsizes may significantly affect the convergence rate of GD, even for quadratic optimization. 
This question of stability to roundoff errors was already raised in Young's original paper~\citep{young53}. In such settings, it is desirable to find permutations of the Chebyshev roots for which GD trajectories are maximally stable.  An effective approach is to interleave the roots of Chebyshev polynomials of increasing degree \citep{LebedevFinogenov71}. This leads to a fractal pattern, superficially similar to our proposed stepsize schedule; see \citep{AgarwalGoelZhang} for a recent discussion and additional results. However, we emphasize that this fractal is not only fundamentally different but also arises due to entirely different considerations---stability rather than fast convergence. 

\paragraph*{Structured quadratics.} If the quadratic function's Hessian has additional spectral structure, then improved results are possible. This is because the different viewpoints discussed above are classically known to extend to this situation via potential theory; see the excellent survey~\citep{6steps} and the references within. This enables further refinements of the methods described above for structured quadratics and sometimes also perturbations away from quadratics; see e.g.,~\citep{oymak2021provable,goujaud2022super}.

\subsubsection{The general case of convex optimization}\label{sssec:intro:convex}

\paragraph*{Unaccelerated GD.} For constant stepsize, the optimal convergence rate for GD is $\Theta(\kappa \log 1/\eps)$ in the strongly convex setting, and $\Theta(1/\eps)$ in the convex setting; see, e.g., the textbooks~\citep{bubeck-book,boyd,nocedal,polyakbook,hazan2016introduction,nesterov-survey}. This is often called the unaccelerated rate for GD. Many alternative stepsize schedules have been proposed in both theory and practice. 
We highlight several well-studied schedules. 
One family of well-studied strategies adaptively chooses stepsizes either by minimizing the function value over the line spanned by the gradient. This mininimization can be performed exactly via line search~\citep{nocedal,boyd,polyakbook,de2017worst}, or approximately via Goldstein-Armijo-type schedules~\citep{nesterov-survey}. Alternatively, it can be done by minimizing the estimated distance to the optimum via Polyak-type schedules~\citep{polyakbook}. Another family is Barzilai-Borwein-type schedules, which are quasi-Newton methods that approximate the Hessian using the past step's change in iterate and gradient \citep{barzilai1988two}. None of these strategies are known to accelerate beyond the case of quadratics.

\paragraph*{Accelerated GD via internal state.} The conventional approach for achieving faster convergence is to consider variations of GD that use auxiliary sequences of iterates and/or different update directions than the gradient. This is of course more powerful than just changing the stepsizes, and can be interpreted from a control theory perspective as adding internal dynamics to the algorithm. Accelerated rates were first shown in Nesterov's seminal work in 1983~\citep{nesterov-agd}, and since then, many other accelerated algorithms and analyses have been proposed~\citep{geometric-gd,quad-averaging, tmm,diakonikolas2017accelerated,taylor23optimalub,cohen2020relative,linear-coupling,beck2009fast,tseng2008accelerated}, as well as fruitful interpretations via continuous-time analysis~\citep{shi2021understanding, su2016differential, shi2019acceleration, WWC16, WRJ16, moucer2022systematic,diakonikolas2021generalized}. These accelerated algorithms require only $\Theta(\sqrt{\kappa} \log1/\eps)$ iterations, or $\Theta(1/\sqrt{\eps})$ in the convex case, which is known to be minimax-optimal up to a constant for any algorithm that uses only gradient information~\citep{nem-yudin}. Much work has recently sharpened this constant, culminating in exactly matching upper and lower bounds~\citep{taylor23optimalub,taylor23optimallb}. This recent line work exploits the idea that the worst-case convergence of optimization algorithms can be numerically computed via semidefinite programming (SDP)~\citep{KF16, DT14, pesto,d2021acceleration,taylor2017convex,tmm,LRP16,DT18,KF17}. 
This has also enabled using computer-automated SDP-analyses to investigate richer classes of algorithms, such as robust versions of accelerated methods~\citep{CHVSL17}, proximal algorithms~\citep{barre2023principled}, operator splitting~\citep{ryu2020operator}, line search~\citep{de2017worst}, biased stochastic gradient methods~\citep{HSL17}, inexact Newton's method~\citep{de2020worst}, among many others. This area of research is extremely active and we refer the reader to the excellent recent survey~\citep{d2021acceleration} for a comprehensive set of references and a detailed historical account.

\paragraph*{Accelerated GD via dynamic stepsizes.} Although many time-varying stepsize schedules have been considered for GD, no convergences analyses improved over the textbook unaccelerated rate beyond the quadratic case. In 2018, Altschuler's MS thesis~\citep{altschuler2018greed} 
considered time-varying stepsize schedules in several settings,
all through the unifying lens of hedging and multi-step descent. In Chapter 8 of the thesis, the PESTO framework was used to show for the first time the advantage of using time-varying stepsize schedules for GD beyond the quadratic setting. Explicit solutions were given for $n=2,3$ in the strongly convex setting. 
This showed that a constant-factor improvement over the textbook unaccelerated GD rate was indeed possible. 
A key difficulty in extending this to larger horizons $n$ is that the search for optimal stepsizes is non-convex. In 2022,~\citet{gupta22} combined Branch \& Bound techniques with the PESTO SDP to develop algorithms that perform this search numerically, and as an example used this to compute good approximate schedules in the convex setting for larger values of $n$ up to $50$.~\citet{grimmer23} very recently developed a technique to round these Branch \& Bound solutions to exact rational certificates. This allowed him to extend these approximate stepsize schedules up to $n=127$ in order to get a larger constant-factor improvement, and conjectured that dynamic stepsizes might lead to an accelerated rate of $O(1/(T\log T))$. By extending a recursive application of the $2$-step solution in~\citep{altschuler2018greed}, the present paper rigorously proves acceleration for all horizons $n$, and in particular obtains the first asymptotic improvements over the textbook unaccelerated GD rate---not just by a constant factor.

\subsection{Organization}\label{ssec:intro:org}

In \S\ref{sec:hedging}, we provide an overview of the core conceptual ideas via the key case $n=2$.~\S\ref{sec:construction} formally defines the Silver Stepsize Schedule and the Silver Convergence Rate $\tau_n$, \S\ref{sec:rate} establishes the claimed properties of $\tau_n$, and \S\ref{sec:cert} proves that $\tau_n$ is a valid bound on the convergence rate of the Silver Stepsize Schedule.~\S\ref{sec:future} discusses future directions. Some technical details are deferred to the Appendix.

\section{Conceptual overview: two-step case ($n=2$)}\label{sec:hedging}

This section provides a complete analysis for the minimal non-trivial horizon length: $n=2$. (No hedging can occur if $n=1$.) Our goal here is to provide further intuition for the core concepts of hedging and multi-step descent, and explain concretely how these manifest in the design and analysis of the Silver Stepsize Schedule. Indeed, the $n=2$ case captures most of the core intuition and ideas, and the result for general $n$ is essentially just an amped-up version thereof. These results first appeared in Altschuler's thesis~\citep[Chapter 8]{altschuler2018greed}; we refer to there for a lengthier treatment.

For simplicity, in this section we denote the stepsizes by $\alpha$ and $\beta$, so that the algorithm is 
\[
x_1 = x_0 - \alpha \nabla f(x_0), \qquad
x_2 = x_1 - \beta \nabla f(x_1)\,,
\]
and the worst-case convergence rate over a function class $\cF$ is
\[
R(\alpha,\beta; \cF) := \sup_{f \in \cF,\; x_0 \neq x^*} \frac{\|x_2 - x^*\|}{\|x_0 - x^*\|}\,.
\]
The question of optimal stepsizes is therefore the minimax problem
\begin{align}
	\min_{\alpha,\beta} R(\alpha,\beta; \cF)\,.
	\label{eq:plausibility:minimax}
\end{align}
To motivate why non-constant stepsizes might be helpful, in \S\ref{ssec:hedging:quad} we first briefly recall the classical result of~\citep{young53} which solves this for the case of quadratic $\cF$.
Then in \S\ref{ssec:hedging:convex}, we solve this problem for convex $\cF$ by presenting the $2$-step Silver Stepsize Schedule from~\citep[Theorem 8.11]{altschuler2018greed}, proving its convergence rate via multi-step descent, and proving its optimality via hedging.

\begin{figure}[t]
	\centering
	\subcaptionbox{\footnotesize Quadratic setting: $(\alpha^*,\beta^*)$ are the two permutations of $\{1.12339, 2.77905 \}$.
		Details in \S\ref{ssec:hedging:quad}.}
	{\includegraphics[width=0.46\textwidth]{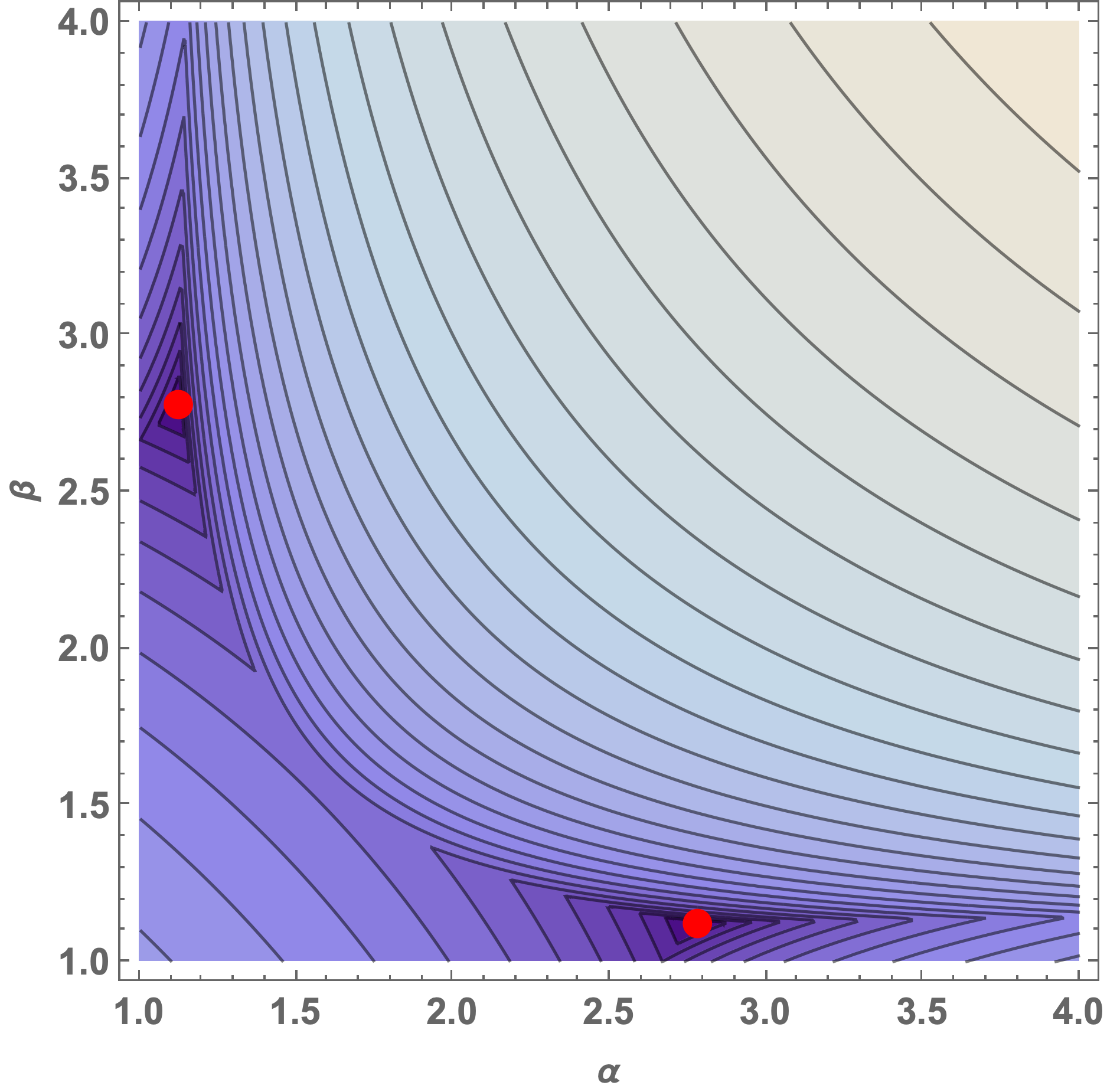}} \hfill
	\subcaptionbox{\footnotesize Convex setting: $(\alpha^*,\beta^*) = (\frac{4}{3},2)$. Details in \S\ref{ssec:hedging:convex}.}
	{\includegraphics[width=0.46\textwidth]{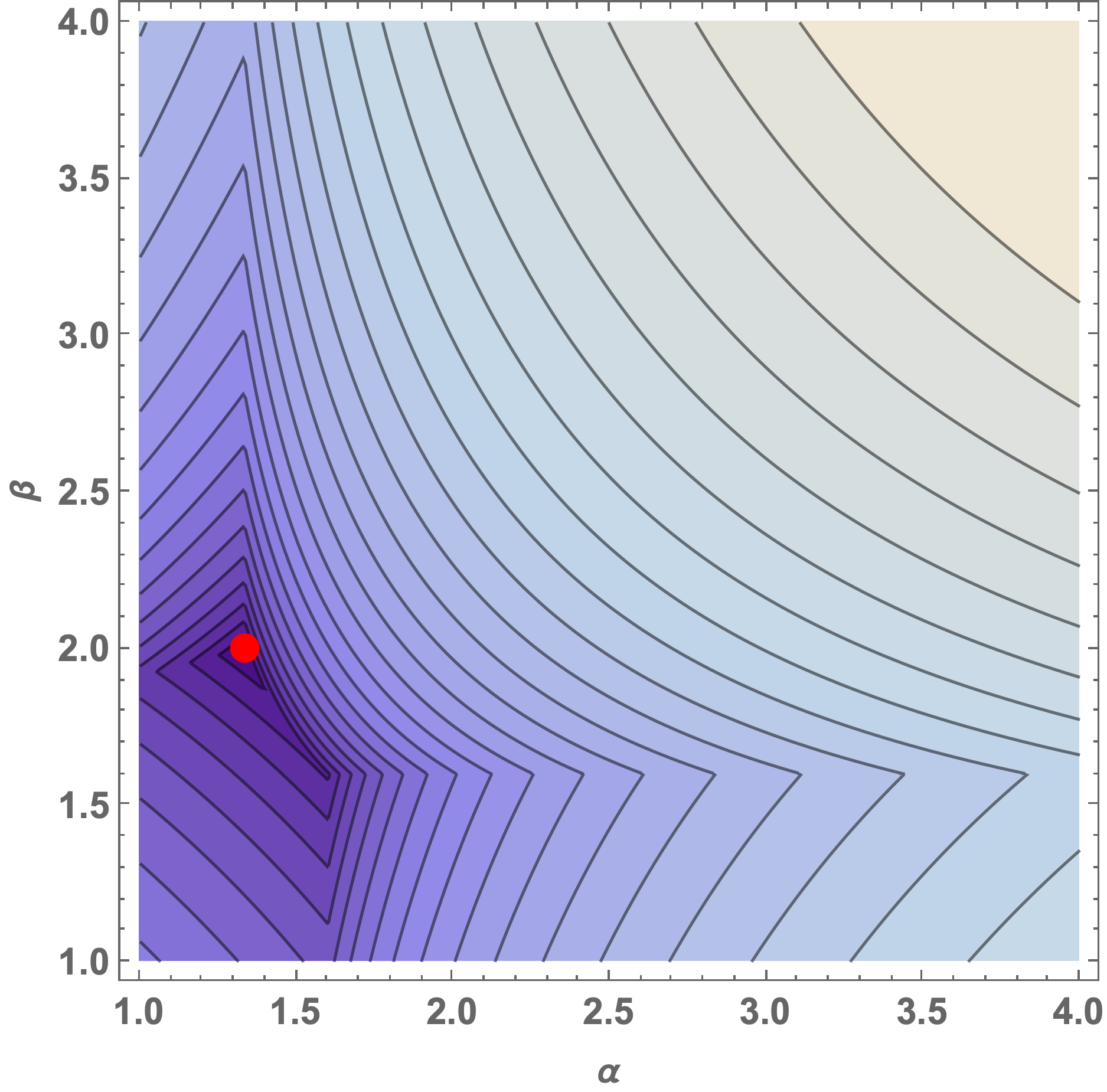}}
	\caption{\footnotesize Contour plots of worst-case rates, as a function of the two stepsizes $\alpha$ and $\beta$, for $m=1/4$ and $M=1$. The marked points indicate the global minima. Notice the asymmetry in the convex case (right), due to the non-commutativity of the GD map.}
	\label{fig:2step}
\end{figure}

\subsection{Optimal stepsizes for quadratic optimization}\label{ssec:hedging:quad}

\paragraph*{Young's argument from 1953.} What is the optimal stepsize schedule $(\alpha,\beta)$ for the class $\cF$ of quadratic functions $f$ that are $m$-strongly convex and $M$-smooth?
Without loss of generality after translating, $f(x) = \frac{1}{2}x^T Hx $ where $mI \preceq H \preceq MI$. By definition of GD, $x_1 = (1 - \alpha H)x_0$ and $x_2 = (1 - \beta H)x_1$, thus
\[
	x_2 = p(H)x_0\,,\qquad \text{ where } \qquad p(H) =(1 - \alpha H)(1 - \beta H).
\]
Observe that as one ranges over all possible choices of the stepsizes $(\alpha,\beta)$, the polynomial $p$ ranges over the set $\cP$ of all degree $2$ polynomials satisfying the normalizing condition $p(0) = 1$. 
Therefore finding optimal stepsizes $(\alpha,\beta)$ is equivalent to finding an optimal polynomial $p \in \cP$. 

\par What is the optimal polynomial? By the above display and properties of the spectral norm,
\begin{align}
	R(\alpha,\beta;\cF)
	=
	\sup_{mI \preceq H \preceq MI, \, x_0 \neq 0} \frac{\|p(H)x_0\|}{\|x_0\|}
	=
	\sup_{mI \preceq H \preceq MI}
	\|p(H)\|
	=
	\sup_{m \leq \lambda \leq M} 
	|p(\lambda)|\,.
	\label{eq:plausibility:quad-pf}
\end{align}
Thus the optimal polynomial $p \in P_2$ is the one with minimal $L_{\infty}$ norm over the interval $[m,M]$. It is classically known that this is the (translated and scaled) Chebyshev polynomial of the first kind, see e.g.,~\citep{Rivlin}. Thus the optimal stepsizes $(\alpha^*, \beta^*)$  are the inverses of the roots $\tfrac{M+m}{2} \plusminus \tfrac{M-m}{2\sqrt{2}}$
of the Chebyshev polynomial, in either order. These are the symmetric marked points in Figure~\ref{fig:2step}, left. 

\par Crucially, observe that these two stepsizes are different---hence the advantage of non-constant schedules in the quadratic setting. We now interpret this phenomenon in two ways that are essential to our intuition for the convex setting. This discussion is based on~\citep[Chapters 1 and 2]{altschuler2018greed}.

\paragraph*{Interpretation via hedging.} Why is $\bar{\alpha} := \tfrac{2}{M+m}$ suboptimal for $2$ steps of GD when it is optimal for $1$? Recall that it is optimal for $1$ step because GD overshoots when using a longer step $\alpha > \bar{\alpha}$ on the sharp function $f(x) = \tfrac{M}{2}x^2$, and undershoots when using a shorter step $\alpha < \bar{\alpha}$ on the shallow function $f(x) = \tfrac{m}{2}x^2$. The algorithmic opportunity is that these worst-case functions are different for short-step GD and long-step GD. This is why using a short step and a long step---each individually suboptimal---can lead to faster overall convergence than using $\bar{\alpha}$ twice. We refer to this misalignment of worst-case functions as \emph{hedging}. See Figure~\ref{fig:2step}, left.

\paragraph*{The necessity of multi-step descent.} There is a dual interpretation of hedging via multi-step descent.
By~\eqref{eq:plausibility:quad-pf}, the worst-case rate for $2$ steps is
\begin{align}
	R(\alpha,\beta;\cF)
	=
	\sup_{m \leq \lambda \leq M} |(1 - \alpha \lambda)( 1- \beta \lambda)|\,.
	\label{eq:plausibility:quad-2step}
\end{align} 
Contrast this with the greedy analysis, which bounds the worst-case rate after $2$ iterations by the product of the worst-case rates for $1$ step with $\alpha$ or $\beta$, namely
\begin{align}
	R(\alpha;\cF) \cdot R(\beta;\cF)
	=
	\left( \sup_{m \leq \lambda_{\alpha} \leq M} |1 - \alpha \lambda_{\alpha}| \right) \cdot \left( \sup_{m \leq \lambda_{\beta}\leq M} |1 - \beta \lambda_{\beta}| \right) 
	\,.
	\label{eq:plausibility:quad-1step}
\end{align}
Observe that the greedy analysis~\eqref{eq:plausibility:quad-1step} is so shortsighted that it not only leads to worse bounds for any given stepsize schedule, but moreover leads to the wrong prescription of stepsizes. Indeed, optimizing this convergence rate~\eqref{eq:plausibility:quad-1step} over $(\alpha,\beta)$ leads to $\alpha = \beta = \tfrac{2}{M+m}$ which is the constant schedule. This necessity of multi-step descent explains why the mainstream approach for convex optimization is constant stepsizes: previous approaches were unable to analyze multi-step descent. (This is only tractable in the quadratic setting because the gradient operator is linear, see~\eqref{eq:plausibility:quad-pf}.)

\subsection{Optimal stepsizes for convex optimization}\label{ssec:hedging:convex}

We now turn to the convex setting. Let $\cF$ denote the set of $m$-strongly convex and $M$-smooth functions. Young's Chebyshev schedule is then provably bad\footnote{This is not just a failure of analysis techniques: even for mild condition numbers like $\kappa = 10$, using the $2$-step Chebyshev Schedule in either order makes GD divergent (i.e., the contraction rate is larger than $1$). We are not aware of a reference for this, but it can be shown e.g., by using the SDP-analysis framework of~\citep{pesto}.}.What are the optimal $2$ stepsizes? Certainly the above discussion of hedging motivates using non-constant stepsizes, but proving this requires multi-step descent, and that has been the longstanding stumbling block preventing progress beyond the quadratic setting.

\subsubsection{Silver Stepsize Schedule for $n=2$}

We show below that the $2$-step convergence rate $ R(\alpha,\beta; \cF)$ is minimized by the stepsizes $(\alpha,\beta)$ that are defined by the system of equations
\begin{equation}
	(M\alpha - 1)(M\beta - 1) = (1 -\alpha m)(1-\beta m) = \frac{(1-m\alpha)(M\beta - 1)}{1+\alpha(M-m)}\,
	\label{eq:twostepsizes}
\end{equation}
and moreover the optimal $2$-step convergence rate $R^*$ is given by this equalized value.

\begin{remark}\label{rem:2step-values}
	The equations~\eqref{eq:twostepsizes} can be solved explicitly, to give the alternative expressions 
	\[
	\alpha^* = \frac{2}{m+S}, \quad
	\beta^* = \frac{2}{2M+m-S}, \quad
	R^* = \frac{S - M}{2 m + S - M},
	\]
	where $S = \sqrt{M^2 + (M-m)^2}$. These are the formulas given in \citep[Thm. 8.10]{altschuler2018greed}, and is the $n=2$ case of the Silver Stepsize Schedule and (square-rooted) Silver Convergence Rate defined in \S\ref{sec:rate}. 
\end{remark}

This $n=2$ solution showcases the key phenomena that also occur for larger $n$:
\begin{itemize}
	\item \textbf{Provable advantage of dynamic stepsizes.} Since $R^* < (\tfrac{M-m}{M+m})^2$, this proves that it is possible to improve over standard GD by dynamically changing the stepsize. (Recall that $\tfrac{M-m}{M+m}$ is the textbook unaccelerated rate for 1 step of GD.) This mirrors how for quadratics, the optimal 2-step rate~\eqref{eq:plausibility:quad-2step} is better than the squared optimal $1$-step rate~\eqref{eq:plausibility:quad-1step}.
	\item \textbf{Stepsize splitting.} Since $\alpha^* < \frac{2}{M+m} < \beta^*$, the optimal stepsize $\frac{2}{M+m}$ for $n=1$ splits into a short step $\alpha^*$ and long step $\beta^*$. For general $n$, the Silver Stepsize Schedule mirrors this splitting at every scale: it splits the largest stepsize into a shorter and longer step.
	\item \textbf{Unique, asymmetric solution.} Unlike the quadratic case, here the stepsize order is essential for fast convergence: the splitting requires the small stepsize to be first.\footnote{We remark that the order may change for different progress measures, see~\citep[Chapter 8.2]{altschuler2018greed}.} As a consequence, here the optimal stepsize schedule is unique. See Figure~\ref{fig:2step}, right. 
	\item \textbf{Milder splitting.} Even ignoring order, the stepsize values differ from the quadratic case. This occurs because the class of convex functions is richer than the class of quadratics, thus the supremum defining the worst-case rate $R(\alpha,\beta; \cF)$ is over more functions, thus it is harder to misalign the worst-cases by hedging. The result is less aggressive hedging and partial acceleration: the improvement over the $1$-step rate is smaller than in the quadratic case. 
\end{itemize}

We now turn to proving that Theorem~\ref{thm:main} holds in the case $n=2$, and moreover that the proposed Silver Stepsize Schedule is optimal among all $2$-step schedules. 

\begin{theorem}[Optimal $2$-step schedule for strongly convex optimization, Theorem 8.11 of~\citep{altschuler2018greed}]\label{thm:2step:convex}
	Consider any strong-convexity and smoothness parameters $0 < m \leq M < \infty$. The unique optimal $2$-step schedule $(\alpha^*,\beta^*) \in \argmin_{\alpha,\beta} R(\alpha,\beta; \cF)$ and the corresponding optimal $2$-step rate $R^*$ are as stated in Remark~\ref{rem:2step-values}.
\end{theorem}

The proof has two parts: an upper bound on $R(\alpha^*,\beta^*, \cF)$ that proves that our $2$-step schedule achieves the claimed rate, and a matching lower bound that proves optimality (and in fact uniqueness too).
 We do this below via multi-step descent and hedging, respectively.

\subsubsection{Upper bound: rate certification via multi-step descent}

As discussed above, in order to prove any benefit of deviating from the constant stepsizes, we must directly analyze the cumulative multi-step descent of all iterations. This requires capturing how different iterations affect other iterations' progress.
We do this by exploiting long-range consistency conditions between the information that GD sees along its trajectory.

\par Our starting point is a known result on convex interpolability, recalled next. There is a set of consistency conditions that any $f \in \cF$ must satisfy at any set of points $\{x_i\}_{i \in \cI}$: the co-coercivity
\begin{align}
	Q(x,y) := 2(M-m)(f(x) - f(y)) + 2\langle M \nabla f(y) - m \nabla f(x), y-x \rangle - \|\nabla f(x) - \nabla f(y)\|^2 - Mm\|x-y\|^2
\nonumber
\end{align}
must be non-negative for every pair of points $x,y \in \cI$. Of particular interest to us is the converse: there are consistency conditions on a set of data $\{(x_i, g_i, f_i)\}_{i \in \cI}$ that ensure it is $\cF$-interpolable, i.e., there exists $f \in \cF$ satisfying $g_i = \nabla f(x_i)$ and $f_i = f(x_i)$ for each $i \in \cI$. Specifically, a celebrated line of work on convex interpolability~\citep{rockafellar} culminated in a beautiful theorem of~\citep{pesto} which states that $\{(x_i, g_i, f_i)\}_{i \in \cI}$ is $\cF$-interpolable if and only if 
\begin{align}
	Q_{ij} := 
        2(M-m)(f_i - f_j) + 2\langle M g_j - m g_i, x_j - x_i \rangle - \|g_i - g_j\|^2 - Mm\|x-y\|^2
	\label{eq:def:Qij}
\end{align}
is non-negative for every pair of indices $i, j \in \cI$.
\par We apply these conditions along the trajectory of GD. Specifically, we take $\cI := \{0, 1, \dots, n, *\}$ to index the GD iterates and the optimum, and let $\{(x_i,g_i,f_i)\}_{i \in \cI}$ denote the first-order data\footnote{This is purely an analysis device and does not change the GD algorithm (which neither knows the optimum nor queries function values). Including function values simplifies the interpolability conditions~\citep{pesto} and thus our analysis.}. The upshot is that this theorem enables replacing the supremum over functions $f \in \cF$ by the data $\{(x_i,g_i,f_i)\}_{i \in \cI}$ in the definition of the worst-case rate $R(\alpha,\beta; \cF)$. Note that this replacement is lossless since the interpolability conditions in the theorem are necessary and sufficient.

\par From the perspective of hedging, these co-coercivity conditions $\{Q_{ij} \geq 0\}_{i \neq j \in \cI}$ generate all possible long-range consistency constraints on the objective function given the GD trajectory. From the perspective of multi-step descent, they generate all possible valid inequalities with which one can prove convergence rates for GD. Let us explain how we use this in the case $n=2$.

\begin{proof}[Proof of rate upper bound for Theorem~\ref{thm:2step:convex}]
	It suffices to prove the \emph{rate certification identity}
	\begin{align}
		R^2 \|x_0 - x^*\|^2 - \|x_2 - x^*\|^2 = \sum_{i \neq j \in \{0, 1, *\}} \lambda_{ij} Q_{ij}
		\label{eq:pf-2step:cert}
	\end{align}
	for some non-negative choice of multipliers $\lam_{ij}$. Indeed, since $Q_{ij} \geq 0$ is non-negative for any objective function $f \in \cF$, the rate certification identity implies
	\begin{align}
		R^2 \|x_0 - x^*\|^2 - \|x_2 - x^*\|^2 \geq 0\,,
	\end{align}
	which proves the claimed rate. It remains to construct non-negative $\lam_{ij}$ for the rate certification identity. This is done in \citep[Theorem 8.10]{altschuler2018greed}. For completeness, we include the explicit values here, in slightly simpler (but equivalent) form:
	\begin{align}
		\lam = \frac{\alpha^* (\beta^*)^2}{4}
            \begin{bmatrix}
			0 & \frac{(S-m)(S-M)}{(M-m)} & 0 \\
			\frac{(S-M) (2M -S- m)}{M-m} 
            & 0 & \frac{2 M^2 + S^2 - 2 M S -m^2}{M-m} \\
			\frac{m^3 - m^2 S + 4 M^2 S - m S^2 - 4 M S^2 + S^3}{M(m+S)}
            & 
            \frac{2 M S - m^2 - S^2}{M-m} & 0 \\
		\end{bmatrix}\,.
	\end{align}
    Here, the rows and columns are indexed by $0, 1,*$, in that order.
\end{proof}

\par Of course, the challenge in such a proof is finding the multipliers $\lam_{ij}$. When we prove our result for general $n$, we prove that the multipliers for the $2n$-length Silver Stepsize Schedule are recursively built from repeating the multipliers for the $n$-length Silver Stepsize Schedule twice, modulo a low rank and sparse correction expressible in closed form. With this recursion, (i) the multipliers for the $n=2$ case above can be derived formulaically from the textbook proof for $n=1$, and (ii) the proof for the case of general $n$ mirrors the proof for $n=2$, at least in spirit.

\subsubsection{Lower bound: optimality and uniqueness via hedging}

\begin{proof}[Proof of rate lower bound in Theorem~\ref{thm:2step:convex}]
	We prove the \emph{rate optimality identity}
	\begin{align}
		R(\alpha,\beta; \cF) \geq \underline{R}(\alpha,\beta)\,,
		\label{eq:pf-2step:opt}
	\end{align}
	for all non-trivial\footnote{We call such stepsizes non-trivial since if a stepsize is outside this interval, then clipping it to the interval improves convergence. This can be proved by noticing that, of the four hard functions in this proof, all but the fourth apply if $\alpha \geq 1/M$, and all but the third apply if $\alpha \leq 1/m$. By optimizing the resulting analogous bounds~\eqref{eq:pf-2step:opt} for the cases $\alpha < 1/M$ and $\alpha > 1/m$, it follows that clipping to the interval $[1/M,1/m]$ leads to faster convergence.
    } stepsizes $\alpha,\beta \in [1/M,1/m]$, where
	\begin{align*}
		\underline{R}(\alpha,\beta) := \max\left\{  (M\alpha-1)(M\beta-1), \, (1 -\alpha m)(1-\beta m), \, (M\alpha-1)(1 - m \beta),\, \frac{(1-m\alpha)(M\beta - 1)}{1+\alpha(M-m)} \right\}\,.
	\end{align*}
	This suffices since it is straightforward to verify that $\min_{\alpha,\beta} \underline{R}(\alpha,\beta)$ is minimized uniquely at $(\alpha^*,\beta^*)$ with value $R^*$; this yields the two defining equations in~\eqref{eq:twostepsizes}. Indeed, this verification can be done by hand by case enumeration, or simpler, it can be rigorously proven using standard symbolic computation techniques such as quantifier elimination~\citep{CavinessJohnson}.
 
	\par It remains to prove~\eqref{eq:pf-2step:opt}. We do this by exhibiting four ``hard-to-optimize'' functions $f \in \cF$ for which the $2$-step convergence rate of GD from initialization $x_0 = 1$ is given by these four values.
    The first two functions are the quadratics $f(x) = \frac{\lam}{2}x^2$ for $\lam \in \{\alpha,\beta\}$, in which case $x_2 = (1-\lam \alpha)(1 - \lam \beta)$. The other two functions are piecewise quadratic. It is perhaps simplest to state these functions via their second derivative since then any function value can be obtained by integrating from the minimum $x^* = 0$. The third function is given by $f''(x) = M$ for $x \geq 0$ and $f''(x) = m$ otherwise, in which case $x_2 = (1-M\alpha)(1 - m\beta)$. The fourth function is given by $f''(x) = m$ for $x \geq  \frac{1-m\alpha}{1+\alpha(M-m)}$ and $f''(x) = M$ otherwise, in which case $x_2 = \frac{(1-m\alpha)(1-M\beta)}{1 + \alpha(M-m)}$. This proves the desired identity~\eqref{eq:pf-2step:opt}.
\end{proof}

\par It is insightful to contrast these four hard functions defining the $2$-step rate function\footnote{The rate optimality identity~\eqref{eq:pf-2step:opt} actually holds with equality over all non-trivial stepsizes $\alpha,\beta \in [1/M,1/m]$, although this is unnecessary for our purposes.} with the analog for the quadratic case. Recall from~\eqref{eq:plausibility:quad-2step} that in the quadratic setting, the $2$-step rate function $R(\alpha,\beta \; \cF) = \sup_{m \leq \lambda \leq M} |(1 - \alpha \lambda)( 1- \beta \lambda)|$. Although this seems to requires infinitely many $\lambda$, it was shown in~\citep[Chapter 5.2.2]{altschuler2018greed} that one can replace the continuum $[m,M]$ with the $3$ extrema of Chebyshev polynomials, i.e., 
\[
	\max_{\lambda \in \{m,M,\frac{M+m}{2}\}} |(1 - \lambda \alpha) (1 - \lambda \beta)|\,,
\]
in the sense that minimizing this over $(\alpha,\beta)$ yields Young's $2$-step Chebyshev Schedule. These three values of $\lambda$ correspond to ``hard-to-optimize'' quadratic functions $f(x) = \tfrac{\lam}{2}x^2$. How are they different from the hard functions in the above proof? The first two quadratics are common between the quadratic and convex case, but the remaining functions differ. In particular, the third and fourth functions in the convex case are non-quadratic. In words, the richness of the convex function class enables changing the curvature in different places, which enables more alignment of the bad convergence rates for the individual stepsizes. This makes it provably harder to hedge in the convex setting. Note also that the denominator of the fourth function is singlehandedly responsible for the asymmetry in $\bar{R}(\alpha,\beta)$, and thus in the optimal schedule the convex setting.

This proof can be extended to establish the optimality of the Silver Stepsize Schedule, as will be detailed in a shortly forthcoming paper.

\section{Silver Stepsize Schedule}
\label{sec:construction}

For simplicity of exposition, from here on we restrict to horizons $n$ that are powers of $2$ (see \S\ref{sssec:intro:dis:schedule} for a discussion of extensions to general $n$), and we set $m=1/\kappa$ and $M=1$ to reduce notational overhead (this is without loss of generality after a possible rescaling).

\subsection{Normalized Silver Stepsizes}\label{ssec:construction:yz}
We construct auxiliary stepsize sequences $y_n,z_n$, that are normalized in a certain way to lie in the interval $[0,1]$. The particular normalization (a certain linear fractional transformation defined in \S\ref{ssec:construction:ab}) simplifies the recursive stepsize splitting by making it a quadratic equation. 
\par Explicitly, initialize the sequences $y_1 = z_1 = 1/\kappa$, and define $y_n,z_n$ recursively from $z_{n/2}$ as the solutions to the defining equations
\begin{align}
	y_nz_n = z_{n/2}^2
	\qquad \text{and} \qquad 
	z_n - y_n = 2(z_{n/2} - z_{n/2}^2)\,.
	\label{eq:yz-defining}
\end{align}
This is the direct analog of the stepsize splitting detailed for the case $n=2$ in \S\ref{ssec:hedging:convex}. Denoting $\xi = 1 - z_{n/2}$, the explicit solution is
\begin{equation}
	y_n = z_{n/2}\, / (\xi + \sqrt{1+\xi^2}) \qquad \text{ and } \qquad
	z_n = z_{n/2}\, (\xi + \sqrt{1+\xi^2})\,.
	\label{eq:yz-recur}
\end{equation}
\begin{figure}[t]
	\centering
	\includegraphics[height=6cm]{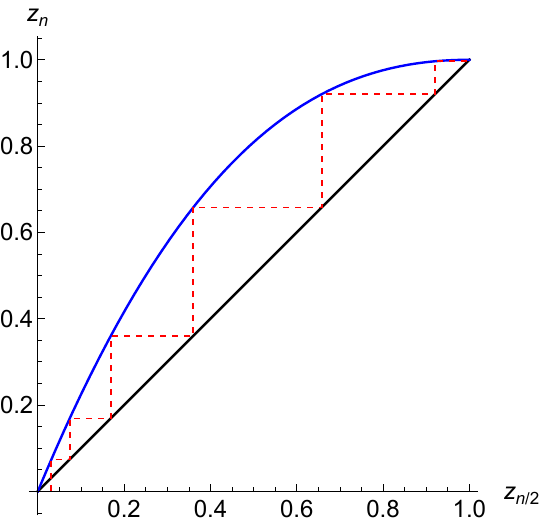}
	\caption{\footnotesize Cobweb plot describing the evolution of $z_n$, under the iteration $z_{n/2} \mapsto z_n$ given in~\eqref{eq:yz-defining} and~\eqref{eq:yz-recur}. The initial condition is $1/\kappa$ (in this plot, $\kappa=32$). The iterates grow exponentially when $z$ is near zero, and converge quadratically to $1$ when $z$ is close to $1$.}
	\label{fig:z-cobweb}
\end{figure}

The following lemma collect several simple observations about these sequences. See \S\ref{sec:rate} for a detailed discussion of how $y_n,z_n$ both increase to their limits $y_n, z_n \to 1$, exponentially fast when they are close to $0$, and then doubly exponentially fast when they are close to $1$.

\begin{lemma}[Basic properties of the Normalized Silver Stepsizes]\label{lem:yz-basic}
	The sequence $z_n$ is monotonically increasing from $z_1 = 1/\kappa$ to $\lim_{n \to \infty} z_n = 1$. 
	For all $n$,
	\begin{align}
		\frac{1}{\kappa} \leq y_n \leq z_n \leq 1\,.
	\end{align}
	Moreover, the above inequality $y_n \leq z_n$ is strict for any $\kappa > 1$.
\end{lemma}

\subsection{Silver Stepsizes}\label{ssec:construction:ab}
We define the Silver Stepsizes
\begin{align}
	a_n := \psi(y_n) \qquad \text{and} \qquad b_n := \psi(z_n) \,.
	\label{eq:yz-reparam-invert}
\end{align}
from the Normalized Silver Stepsizes $y_n,z_n$ via the linear fractional transformation $\psi$ given by
\begin{align}
    \psi : t \mapsto \frac{1 + \kappa t}{1 + t}
    \qquad \text{and} \qquad 
    \psi^{-1} : s \mapsto \frac{s-1}{\kappa-s}\,.
\end{align}

\par We remark that this mapping $\psi$ has the following special values
\begin{align}
	\psi(0)= 1
	\,, 
	\qquad \psi(1/\kappa) = \frac{2}{1+1/\kappa}\,, \qquad \psi(1) = \frac{1+\kappa}{2}\,, \qquad \psi(\infty)= \kappa \,.
	\label{eq:psi-special}
\end{align}
The significance of the two middle values is that these are the initial stepsizes $a_1 = b_1 = \psi(1/\kappa)$ and the limiting stepsizes $\lim_{n \to \infty} a_n = \lim_{n \to \infty} b_n = \psi(1)$. We remark that these two middle values are the harmonic and arithmetic means of the two extremal values, i.e., $\psi(1/\kappa) = \mathrm{HM}(\psi(0), \psi(\infty))$ and $\psi(1) = \mathrm{AM}(\psi(0), \psi(\infty))$. 
 The looseness in the classical AM-HM inequality therefore quantifies the gap between the initial and limiting stepsizes. The following lemma records this and several other simple observations about these stepsizes. 

\begin{lemma}[Basic properties of the Silver Stepsizes]\label{lem:ab-basic}
	The sequence $b_n$ is monotonically increasing from $b_1 = \mathrm{HM}(1,\kappa)$ to $\lim_{n \to \infty} b_n = \mathrm{AM}(1,\kappa)$. 
	For all $n$,
	\begin{align}
		1 \leq \mathrm{HM}(1,\kappa) \leq a_n \leq b_n \leq \mathrm{AM}(1,\kappa) \leq \kappa \,.
	\end{align}
	Moreover, the above inequality $a_n \leq b_n$ is strict for any $\kappa > 1$.
\end{lemma}

\subsection{Silver Stepsize Schedule}\label{ssec:construction:schedule}

\begin{figure}[t]
	\centering
 	\includegraphics[width=.24\columnwidth]{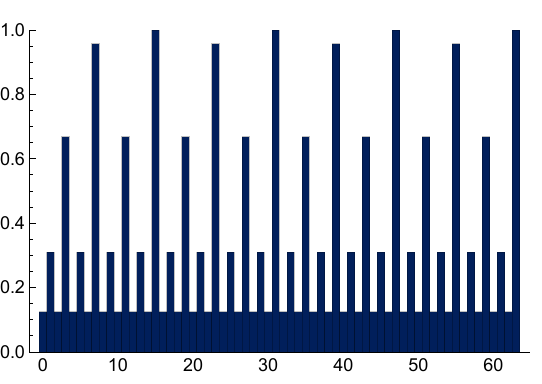}
   	\includegraphics[width=.24\columnwidth]{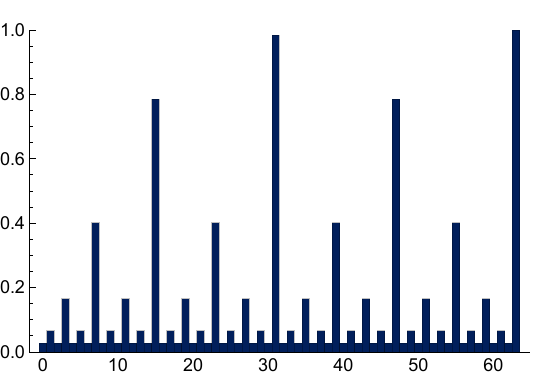}
 	\includegraphics[width=.24\columnwidth]{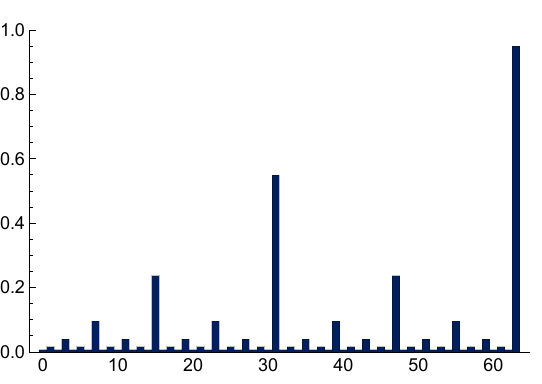}
 	\includegraphics[width=.24\columnwidth]{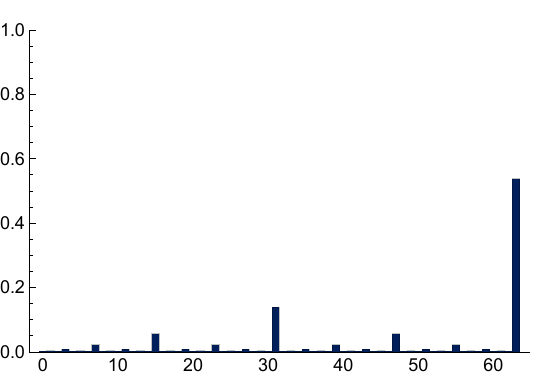}

	\caption{\footnotesize Normalized Silver Stepsize Schedule, for different condition numbers $\kappa = 4,16,64,256$. Notice that these are always bounded between 0 and 1. The Silver Stepsize Schedules $h^{(n)}$ shown in Figure~\ref{fig:stepsizes} are generated by applying $\psi$ to the schedules here.}
	\label{fig:normalizedstepsizes}
\end{figure}

Let $h^{(n)}$ denote the Silver Stepsize Schedule of length $n$.
Denote its $n/2$-th stepsize by $a_n$ and its $n$-th by $b_n$. As overviewed briefly in \S\ref{sssec:intro:dis:schedule}, we recursively construct
\begin{align}
	h^{(n)} := [\tilde{h}^{(n/2)}, a_n, \tilde{h}^{(n/2)}, b_n]\,,
	\label{eq:recurrence-stepsizes}
\end{align}
where $\tilde{h}^{(n/2)}$ denotes everything in $h^{(n/2)}$ except the final step, i.e., everything except $b_{n/2}$. Note that $b_{n/2}$ is in $h^{(n/2)}$, but not in $h^{(n)}$; it is split into $a_n$ and $b_n$. Note also that $a_n$, $b_n$ form the largest stepsizes in $h^{(n)}$, with $b_n$ being the largest (Lemma~\ref{lem:ab-basic}). For the convenience of the reader, we recall from \S\ref{sssec:intro:dis:schedule} that for small $n$, this pattern is
\begin{align*}
	h^{(1)} &= [a_1] \\
	h^{(2)} &= [a_2, b_2] \\
	h^{(4)} &= [a_2,a_4,a_2,b_4] \\
	h^{(8)} &= [a_2,a_4,a_2,a_8,a_2,a_4,a_2,b_8]
\end{align*}
See Figure~\ref{fig:stepsizes} for an illustration of this pattern, and see \S\ref{sssec:intro:dis:schedule} for a discussion of the emergent fractal, dependence on the horizon, and patterns for small $n$. 

\begin{remark}[Occupation measure]
	For all $i \in \N$ and all sufficiently large horizons $n \geq 2^i$, the stepsize $a_{2^i}$ is used in $2^{-i}$ fraction of the $n$-step Silver Stepsize Schedule.
	For example, for all horizons $n \geq 2$, the smallest stepsize $a_2 = \kappa/(\kappa - 1)$ is used in every other iteration. For the infinite limit of the Silver Stepsize Schedule (see \S\ref{sssec:intro:dis:schedule}), the occupation measure simplifies to 
	\begin{align*}
		\sum_{i = 1}^{\infty} 2^{-i} \delta_{a_{2^i}}\,.
	\end{align*}
	This can be viewed as a geometric distribution that takes value $a_{2^i}$ with probability $2^{-i}$.
\end{remark}

\subsection{Silver Convergence Rate}  

We define the Silver Convergence Rate as
	\begin{align}
		\tau_n := \left(\frac{1-z_n}{1+z_n} \right)^2\,.
		\label{def:tau}
	\end{align}
Of course, from just this definition it is not yet clear why we call $\tau_n$ a rate; in \S\ref{sec:cert} we prove that $\tau_n$ is the convergence rate of the Silver Stepsize Schedule. Note that since $z_n$ is monotonically increasing (Lemma~\ref{lem:yz-basic}), this rate $\tau_n$ is monotonically decreasing from the textbook unaccelerated rate $\tau_1 = ((\kap - 1)/(\kap+ 1))^2$ to $\lim_{n \to \infty} \tau_n = 0$. In the following section, we provide a complete understanding of exactly how fast $\tau_n$ converges to $0$.

\section{Analysis of the Silver Convergence Rate}\label{sec:rate}

Here we prove the bound on the Silver Convergence Rate $\tau_n$ in our main result (Theorem~\ref{thm:main}). We restate this bound for convenience.

\begin{theorem}[Silver Convergence Rate]\label{thm:rate}
	Denote $i^* := \lfloor \log_{\rho} \frac{\kappa}{3} \rfloor$ and $n^* := 2^{i^*}$. Then for any $n$ that is a power of $2$, we have the following bound on $\tau_n$. 
	\begin{itemize}
		\item \underline{Acceleration regime.} If $n \leq n^*$, then 
		$$\tau_n = 	\exp\left( - \Theta\left( \frac{n^{\log_2 \rho}}{\kappa} \right) \right)\,. $$
		\item \underline{Saturation regime.} If $n > n^*$, then $$\tau_n = \exp\left( - \Theta\left(  \frac{n}{ n^*} \right) \right)\,.$$
	\end{itemize}
\end{theorem}

This result establishes $n^* \asymp \kappa^{\log_{\rho} 2}$ as the location of a phase transition. There, the Silver Convergence Rate $\tau_n$ switches from super-exponential to exponential in the horizon $n$. See the introduction for a detailed discussion of this phase transition and the intuition behind it in terms of how the Silver Stepsize Schedule is effectively periodic with periodic of length $n^*$.

\par For simplicity, we make no attempt to optimize the constants in the $\Theta$ and the choice of $i^*$ (the $1/3$ in the theorem statement is arbitrary). 
Our proofs make crude constant bounds to ease the exposition, and it is straightforward to tighten these. However, as established by our upper and lower bounds, our proofs are already tight up to reasonable constant factors.

\par The section is organized as follows. In \S\ref{ssec:rate:heuristic}, we provide a heuristic derivation of Theorem~\ref{thm:rate} that explains the phase transition via Taylor expanding the dynamics in the two regimes. This gives the central intuition for the result and its proof. In \S\ref{ssec:rate:rigorous}, we make these Taylor expansions precise to conclude the proof of Theorem~\ref{thm:rate}.

\subsection{Heuristic derivation}\label{ssec:rate:heuristic}

The phase transition in $\tau_n = (\tfrac{1-z_n}{1+z_n})^2$ is a consequence of the phase transition in the dynamics of the auxiliary sequence $z_n$.
To explain this, it is convenient to simplify notation by re-indexing $n = 2^i$ so that iterations of the dynamical process are indexed by $i=0,1,2,3,\dots$ rather than $n=1,2,4,8\dots$. It is helpful to also re-parameterize 
\begin{align*}
	h_i := \Psi(z_{2^i})\,,
\end{align*}
where $\Psi : (0,1) \to (0,1)$ is the monotone bijection
\begin{align*}
	\Psi(z) := \frac{2z}{1+z}\,.
\end{align*}
The significance of this re-parameterization to $h_i$ is that
\begin{align}
	\tau_n = \left( \frac{1-z_n}{1+z_n} \right)^2 = (1 - h_i)^2 \,.
	\label{eq:rate:tau-h}
\end{align}
Thus, proving a fast convergence rate amounts to lower bounding $h_i$.

\begin{figure}[t]
	\centering
	\includegraphics*[height=7cm]{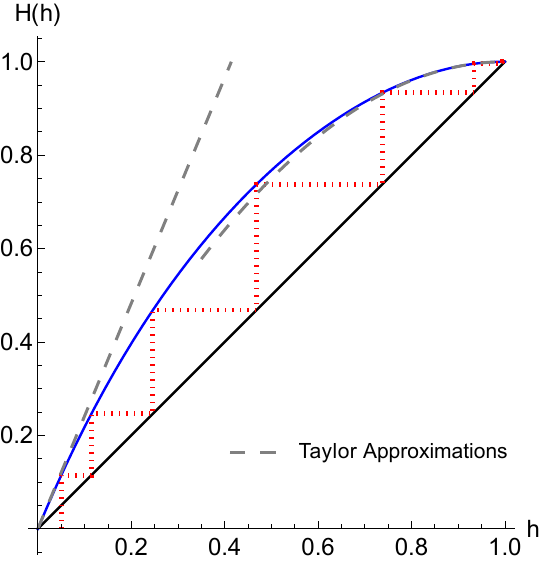}
	\caption{\footnotesize Cobweb diagram for the function $H(h)$ in equation~\eqref{eq:rate:def-H} and its Taylor approximations~\eqref{eq:regime-acceleration:taylor} and~\eqref{eq:regime-saturation:taylor}.  Starting from the initial condition~\eqref{eq:pf-rate:init}, iterates grow by a constant multiplicative factor of the Silver Ratio $\rho = 1 + \sqrt{2}$ when $h$ is near zero, and converge quadratically to 1 when $h$ is close to 1. 
    }
	\label{fig:h-iteration}
\end{figure}

\par What do the dynamics of $h_i$ look like? At initialization, $z_1 = 1/\kappa$ (see \S\ref{sec:construction}), thus
\begin{align}
	h_0 = \frac{2}{1 + \kappa} \asymp \frac{1}{\kappa}\,.
	\label{eq:pf-rate:init}
\end{align}
Then the iterations of this process increase $h_i$ exponentially fast to $1$ when it is sub-constant size, and then doubly-exponentially fast when $h_i$ is of constant size. 
(This dichotomy is the source of the phase transition in Theorem~\ref{thm:rate}.)
\par To analyze these dynamics, let $H : (0, 1) \to (0,1)$ denote the update function sending $h_i$ to $h_{i+1}$. Then $H = \Psi \circ F \circ \Psi^{-1}$ where $F(z) = z (1 - z + \sqrt{1 + (1-z)^2})$ is the function that updates $z_n$ to $z_{2n} = F(z_n)$, see Section~\ref{sec:construction}. A direct algebraic computation gives the explicit expression
\begin{align}
	H(h) = \frac{h \, (2 - 3 h + \sqrt{5 h^2 - 12 h + 8})}
    {2(1-h^2)}\,.
	\label{eq:rate:def-H}
\end{align}
Taylor expanding $H$ around $h \approx 0$ and $h \approx 1$ illustrates the markedly different dynamics in these two regimes; 
see Figure~\ref{fig:h-iteration}.
\begin{itemize}
	\item \underline{Acceleration regime.} For $h \ll 1$, 
	\begin{align}
		H(h) \approx \rho h\,.
		\label{eq:regime-acceleration:taylor}
	\end{align}
	Thus, in this regime, each $h_i$ increases by a factor of roughly $\rho$, thus $h_i \approx \rho^i h_0 \approx \rho^i / (2\kap)$, thus the Silver Convergence Rate is roughly
	\begin{align}
		\tau_n = (1 - h_i)^2 \approx \exp( - 2 h_i) \approx \exp\left( -\rho^i / \kappa\right)
		= \exp\left( - n^{\log_2 \rho} / \kappa \right)
		\,.
		\label{eq:informal-rate:acceleration}
	\end{align}
	This regime lasts for only $i \approx \log_{\rho} \kappa$ iterations (aka horizon $n=2^i \approx \kappa^{\log_2 \rho}$)
	because at that point $h_i \asymp \rho^i / \kap \asymp 1$ is of constant size. This is the phase transition.
	\item \underline{Saturation regime.} For $h \approx 1$, 
	\begin{align}
		1 - H(h) \approx (1 - h)^2\,.
		\label{eq:regime-saturation:taylor}
	\end{align}
	In words, the key phenomenon here is that the average rate $\tau_n^{1/n}$ stays essentially the same as $n$ increases---in contrast to the acceleration regime, in which the average rate improves in $n$. Indeed, the Taylor expansion~\eqref{eq:regime-saturation:taylor} indicates that in the saturation regime, $\tau_{n} = \left( 1 - h_{i}\right)^2 \approx (1 - h_{i-1})^4 = \tau_{n/2}^2$. 
	By repeating this argument and then using the fact that $\tau_{n^*} = \exp(-\Theta(1))$ which follows from the acceleration regime, we obtain
	\begin{align}
		\tau_n \approx \left( \tau_{n^*} \right)^{n / n^*}
		= \exp\left( - \Theta\left( n/n^* \right) \right)\,.
	\end{align}
\end{itemize}

\par If the approximations were justified in the above two displays, then this informal argument would lead to a proof of Theorem~\ref{thm:rate}. We do this in the following subsection.

\subsection{Rigorous derivation}\label{ssec:rate:rigorous}

Here we prove Theorem~\ref{thm:rate}. 
We first state two helper lemmas, which formalize the Taylor approximations~\eqref{eq:regime-acceleration:taylor} and~\eqref{eq:regime-saturation:taylor} in the acceleration regime and saturation regime, respectively.

\begin{lemma}[Dynamics in the acceleration regime]\label{lem:dynamics-acceleration}
	Let $\nu := \frac{3\rho}{2\sqrt{2}} \approx 2.561$. For all $h \geq 0$,
	\begin{align*}
		\rho h - \nu h^2 \leq H(h) \leq \rho h \,.
	\end{align*}
\end{lemma}

\begin{lemma}[Dynamics in the saturation regime]\label{lem:dynamics-saturation}
	For all $h \geq 0$,
	\begin{align*}
		(1-h)^2 - (1-h)^4 \leq 1 - H(h) \leq (1 - h)^2\,.
	\end{align*}
\end{lemma}
We omit the proofs of these lemmas, since the inequalities are visually obvious from plotting the functions, and can be formally proven in a routine algorithmic way, as they only involve algebraic functions of a single scalar variable. This is done by computing the critical points and using well-known techniques for root isolation; see e.g.~\citep{BPRbook}.

An appealing consequence of Lemma~\ref{lem:dynamics-saturation} is the inequality $\tau_{2n} \leq \tau_n^2$. We call this the \emph{rate monotonicity} property of the Silver Stepsize Schedule, since it amounts to the statement that using the $2n$-step schedule is at least as good as using the $n$-step schedule twice.

\begin{cor}[Rate monotonicity property for the Silver Stepsize Schedule]\label{cor:rate-monotonicity}
	For any $n$ that is a power of $2$,
	\[
		\tau_{2n} \leq \tau_{n}^2\,.
	\]
\end{cor}
\begin{proof}
	By~\eqref{eq:rate:tau-h}, then Lemma~\ref{lem:dynamics-saturation}, then~\eqref{eq:rate:tau-h} again, we have
	$\tau_{2n} = (1 - h_{i+1})^2 \leq (1-h_i)^4 = \tau_n$.
\end{proof}

\begin{proof}[Proof of Theorem~\ref{thm:rate}] Here we prove the upper bounds (a.k.a., the convergence rates). The matching lower bounds are conceptually identical and deferred to Appendix~\ref{app:rate} for brevity.
\par \underline{Acceleration regime.} Suppose $n \leq n^*$. Let $i := \log_2 n \leq i^*$. We bound
	\begin{align}
		h_i 
		\geq 
		\rho^i h_0 - \nu \sum_{t=0}^{i-1} \rho^{i-1-t} h_t^2
		\geq \rho^i h_0 - \nu h_0^2 \rho^{i-1} \sum_{t=0}^{i-1} \rho^t
		\geq \rho^i h_0 \left( 1 - 
		\frac{3}{4}
		\rho^i h_0 \right)
		\geq \frac{\rho^i h_0}{2}
		=
		\frac{n^{\log_2 \rho}}{2\kappa}\,.
		\label{eq:pf-rate:1}
	\end{align}
	Above, the first step is by $i$ applications of the lower bound in Lemma~\ref{lem:dynamics-acceleration}. The second step is because $h_t \leq \rho^t h_0$ by $t$ applications of the upper bound in Lemma~\ref{lem:dynamics-acceleration}. The third step is by summing the geometric series, crudely dropping a positive term, and simplifying $\nu/(\rho\sqrt{2}) = 3/4$. The fourth step is because $\rho^i h_0 \leq \rho^{i^*} h_0 \leq 2/3$ by definition of $i^*$ and the initialization upper bound $h_0 \leq 2/\kappa$, see~\eqref{eq:pf-rate:init}. The final step is by definition of $i = \log_2 n$ and the initialization lower bound $h_0 \geq 1/\kappa$, see~\eqref{eq:pf-rate:init}. This completes the proof since by~\eqref{eq:rate:tau-h},
	\begin{align*}
		\tau_n = (1 - h_i)^2 \leq \exp( -2 h_i) \leq \exp(- n^{\log_2 \rho} / \kappa)\,.
	\end{align*}
	
	\par \underline{Saturation regime.} Next, suppose $n > n^*$. By $n/n^*$ applications of Corollary~\ref{cor:rate-monotonicity} and then using the bound on $\tau_{n^*}$ proved in the acceleration regime, we have
	\begin{align*}
		\tau_n
		\leq
		\left( \tau_{n^*} \right)^{n / n^*} 
		\leq 
		\exp\left( - \frac{n}{n^*} \cdot \frac{(n^*)^{\log_2 \rho}}{\kappa} \right)
		\,.
	\end{align*}
	The proof is complete by using the definition of $n^*$ and $i^*$ to bound $(n^*)^{\log_2 \rho} = \rho^{i^*} \geq \frac{\kappa}{3\rho}$.
\end{proof}

\section{Certificate of the Silver Convergence Rate}\label{sec:cert}

Here we prove that the Silver Stepsize Schedule has convergence rate $\tau_n$. This is where we establish multi-step descent. 
For a conceptual overview, we refer the reader to \S\ref{ssec:hedging:convex} for the case of $n=2$; the proof for general $n$ here mirrors that key case, albeit is more technically involved. 

\par Recall from the discussion there that the proof strategy amounts to finding a \emph{certificate} $\{\lam_{ij}\}$ for the rate $\tau_n$, by which we mean non-negative multipliers $\{\lam_{ij}\}_{i,j \in \{0, \dots, n-1,*\}}$ such that 
\begin{align}
	\tau_n \|x_0 -  x^*\|^2 - \|x_n - x^*\|^2 = \sum_{i,j \in \{0, 1, \dots, n, *\}} \lam_{ij}Q_{ij}\,.
\end{align}
See~\S\ref{ssec:hedging:convex} for a definition of the co-coercivities $Q_{ij}$. Briefly, these are valid inequalities that generate all possible long-range consistency conditions between the gradients seen along GD's trajectory. 

\begin{figure}[t]
	\centering
	\includegraphics[width=0.5\linewidth]{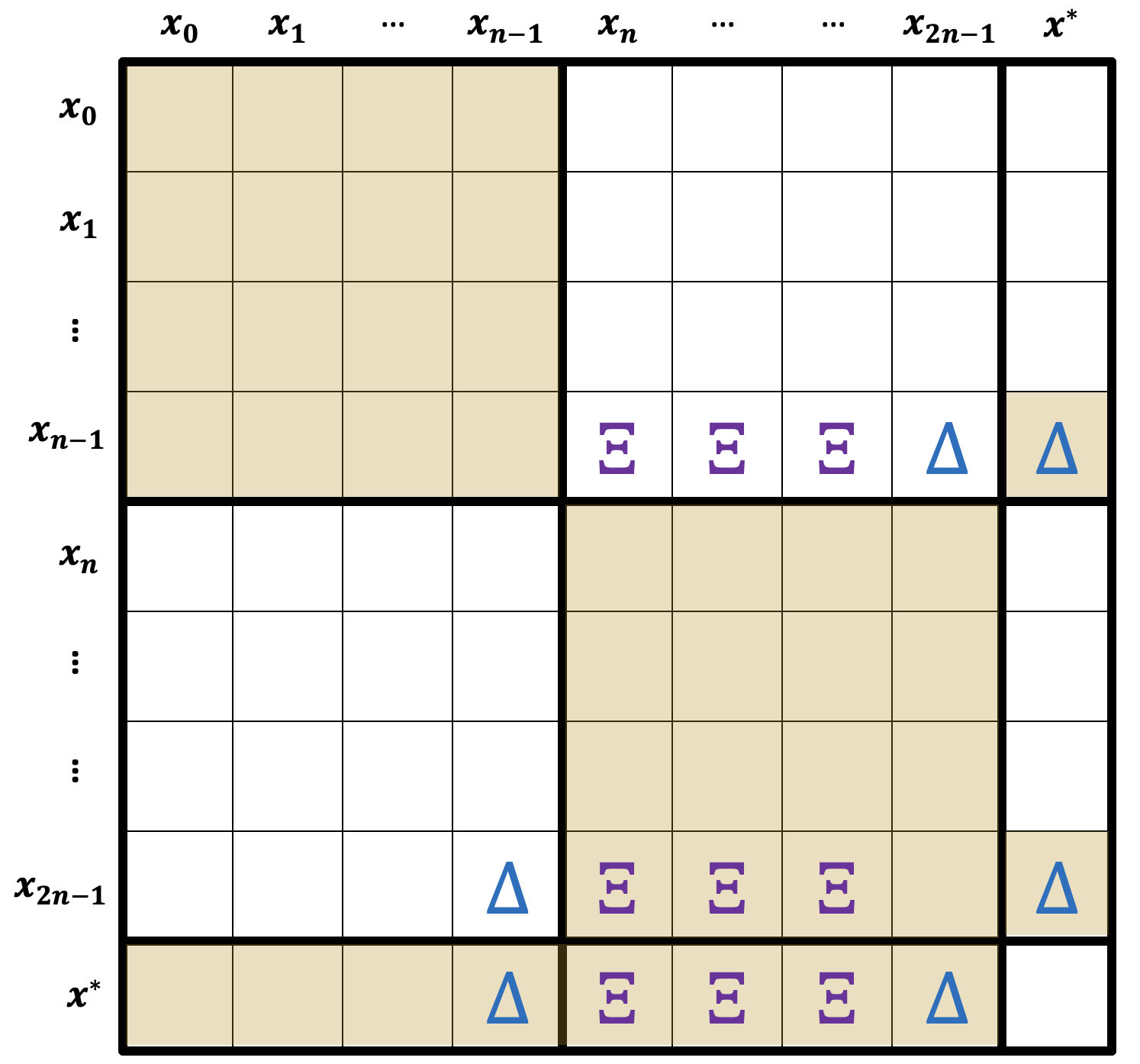}
	\caption{\footnotesize Components of the recursively glued certificate in Theorem~\ref{thm:cert}, illustrated here for combining two copies of the $n=4$ certificate (shaded) to create the $2n=8$ certificate.}
	\label{fig:gluing}
\end{figure}

\par Our proof builds the $2n$-step certificate by \emph{recursively gluing} two copies of the $n$-step certificate and adding slight modifications to account for the fact that the $2n$-step Silver Stepsize Schedule $h^{(2n)}$ differs from $[h^{(n)}, h^{(n)}]$ in two out of the $2n$ stepsizes. Concretely, this recursive gluing can be understood as creating the $(2n+1) \times (2n+1)$ matrix $\{\lam_{ij}\}_{i,j \in \{0, \dots, 2n-1, *\}}$ from the $(n+1) \times (n+1)$ matrix $\{\sig_{ij}\}_{i,j \in \{0, \dots, n-1,*\}}$ in three parts: \textit{a tensor product} which glues together two copies of the $n$-step certificate, a \textit{rank-one correction} which affects the rows indexed by $i \in \{n-1,2n-1,*\}$, and a \textit{sparse correction} which affects the $6$ entries $(i,j)$ where $i \neq j \in \{n-1,2n-1,*\}$. See Figure~\ref{fig:gluing}.

This recursive gluing is formally stated in Theorem~\ref{thm:cert} below. To most easily state this result, we first isolate a certain property of the sparsity pattern of the multipliers $\lam_{ij}$ that holds by construction in our recursion. This property is technical and eases the proof.

\begin{defin}[$*$-sparsity property]\label{def:*sparsity}
	A collection of weights $\{\lam_{i,j}\}_{i,j \in \{0, \dots, n-1,*\}}$ satisfies the \emph{$*$-sparsity property} if $\lam_{i,*}$ for all $i < n-1$. 
\end{defin}	

\begin{theorem}[Recursive gluing for the Silver Stepsize Schedule]\label{thm:cert}
	Let $\kappa \in (1,2) \cup (2, \infty)$. Suppose $\{\sigma_{ij}\}_{i,j \in \{0, \dots, n-1, *\}}$ satisfies $*$-sparsity and certifies the $n$-step rate, i.e.,
	\begin{align}
		\tau_n \|x_0 - x^*\|^2 - \|x_n - x^*\|^2 
		=
		\sum_{i,j \in \{0, \dots, n-1, *\}} \sig_{ij} Q_{ij}
		\, \qquad \text{ for stepsize schedule } h^{(n)}\,.
		\label{eq:cert:n}
	\end{align}
	Then there exists $\{\lam_{ij}\}_{i,j \in \{0, \dots, 2n-1, *\}}$ that satisfies $*$-sparsity and certifies the $2n$-step rate, i.e.,
	\begin{align}
		\tau_{2n} \|x_0 - x^*\|^2 - \|x_{2n} - x^*\|^2 =
		\sum_{i,j \in \{0, \dots, 2n-1, *\}} \lam_{ij} Q_{ij}\, \qquad \text{ for stepsize schedule } h^{(2n)}\,.
		\label{eq:cert:2n}
	\end{align}
	Moreover, this certificate is explicitly given by
	\begin{align}
		\lam_{ij} := \underbrace{\Theta_{ij}}_{\text{gluing}} + \underbrace{\Xi_{ij}}_{\text{rank-one correction}} + \underbrace{\Delta_{ij}}_{\text{sparse correction}}
		\label{eq:recurrence:construction}
	\end{align}
	where the ``gluing component'' $\Theta$ is defined as 
	\begin{align}
		\Theta_{i,j} := \underbrace{\frac{\tau_{2n}}{\tau_n} \sig_{i,j} \cdot \mathds{1}_{i,j \in \{0, \dots, n-1,*\}}}_{\text{recurrence for first $n$ steps}}
		\;+ \underbrace{c \sig_{i-n,j-n} \cdot  \mathds{1}_{i,j \in \{n,\dots,2n-1,*\}}}_{\text{recurrence for second $n$ steps}}\,,
	\end{align}
	the ``rank-one correction'' $\Xi$ is zero except $\{\Xi_{ij}\}_{i \in \{n-1,2n-1,*\}, j \in \{n, \dots, 2n-2\}}$, and the ``sparse correction'' $\Delta$ is zero except $\{\Delta_{ij}\}_{i \neq j \in \{n-1,2n-1,*\}}$. The explicit values of $c$, $\Xi$, $\Delta$ are provided in Appendix~\ref{app:cert}.
\end{theorem}

While the explicit values of $\Xi$ and $\Delta$ are somewhat involved, the key point is that they can be expressed as rational functions in just $z_n, y_{2n}, z_{2n}$, see Remark~\ref{rem:simple-yz}. Importantly, since $y_{2n}, z_{2n}$ are explicit algebraic functions of $z_n$ by construction (see \S\ref{ssec:construction:yz}), this turns verifying the claimed identity~\eqref{eq:cert:2n} into a straightforward (albeit tedious) algebraic exercise that is rigorously automatable via standard computer algebra techniques~\citep{cox2013ideals}.

\par A few minor remarks. First, for simplicity, Theorem~\ref{thm:cert} assumes $\kappa \neq 2$. This allows us to multiply and divide by $\kap-2$, which simplifies expressions. 
The rate for $\kappa = 2$ anyways follows immediately from $\kappa = 2+\eps$ for $\eps \downarrow 0$. Second, in~\eqref{eq:recurrence:construction} the notational shorthand $i-n$ is understood to be $*$ when $i=*$.
Third, when it is said that $\lam$ satisfies $*$-sparsity, it is understood that this property corresponds to the horizon of length $2n$, i.e., $\lam_{*,i} = 0$ for all $i < 2n-1$.

Theorem~\ref{thm:cert} immediately implies the convergence rate~\eqref{eq:thm-main:cert} in our main result Theorem~\ref{thm:main}. 

\begin{proof}[Proof of convergence rate in Theorem~\ref{thm:main}]
	The base case of $n=1$ is the classical analysis of GD; see, e.g.,~\citep[Chapter 8]{altschuler2018greed} for a proof in this language of co-coercivities.\footnote{One can also use Theorem~\ref{thm:2step:convex} to take $n=2$ as the base case. Then this paper's proof is fully self-contained.}
	The convergence rate is the textbook unaccelerated rate $\tau_1 = (\tfrac{\kappa - 1}{\kappa + 1})^2$. By induction, Theorem~\ref{thm:2step:convex} implies that $\tau_n$ is a valid convergence rate for the $n$-step Silver Stepsize Schedule, for all $n$ that are powers of $2$.
\end{proof}

Below, in \S\ref{ssec:rate:corrections}, we express the components of the recursively glued certificate as succinct quadratic forms, and then in \S\ref{ssec:rate:pf}, we use this to prove Theorem~\ref{thm:cert}.

\subsection{Recursive gluing}\label{ssec:rate:corrections}

Proving Theorem~\ref{thm:cert} requires establishing that $\lam$ satisfies the identity~\eqref{eq:cert:2n}. Ignoring presently the linear form in the function values (that term is much simpler and addressed in \S\ref{ssec:rate:pf}), this amounts to showing equality of two quadratic forms. Na\"ively, this requires checking equality of \emph{all} coefficients of these quadratic forms---which is painstaking since these are quadratics in all the GD iterates $x_0, \dots, x_{2n-1}, x^*$ and their corresponding gradients $g_0, \dots, g_{2n-1}, g^*$, and moreover are defined over the ideal generated by the GD equations $x_{t+1} = x_t - \alpha_t g_t$. A key observation that removes much of this labor is that \emph{the quadratic forms in our recursive certificate have rank at most $4$.} In fact, these quadratic forms are only in the four variables $x_{n-1}, g_{n-1}, x_{2n-1}, g_{2n-1}$. This reduces the number of coefficients to be checked from $\Theta(n^2)$ to a constant number: $10$. 
\par This observation is formalized in the following lemma, which expresses the quadratic forms via coefficient matrices as this is convenient for book-keeping. For brevity, just as in Theorem~\ref{thm:cert}, the explicit values of these matrices are deferred to the Appendix, but the key point is that each entry can be expressed a rational function of just $z_n, y_{2n}, z_{2n}$, see Remark~\ref{rem:simple-yz}. To isolate the quadratic form component of the co-coercivities, let $P_{ij}$ denote $Q_{ij}$ without its linear component $f_i - f_j$, i.e.,
\begin{align*}
	P_{ij} := 2\langle g_j - \frac{g_i}{\kappa}, g_j - g_i \rangle - \|g_i - g_j\|^2 - \frac{1}{2(\kappa - 1)}\|x_i - x_j\|^2\,.
\end{align*}

\begin{lemma}[Recursive gluing via succinct quadratic forms]\label{lem:succinct} Consider the setup of Theorem~\ref{thm:cert}, let $v := [x_{n-1}, g_{n-1}, x_{2n-1}, g_{2n-1}]^T$, and let $E$, $S$, $L$ be the $4 \times 4$ matrices defined in Appendix~\ref{app:succinct}.
	\begin{itemize}
		\item \underline{Gluing error:} $ \tau_{2n} \|x_0\|^2 - \|x_{2n}\|^2 - \sum_{i,j \in \{0,\dots,2n-1,*\}} \Theta_{ij} P_{ij} = \langle E, vv^T \rangle$
		\item \underline{Sparse correction:} $\sum_{ij} \Delta_{ij}P_{ij} = \langle S, vv^T \rangle$
		\item \underline{Rank-one correction:} $\sum_{ij} \Xi_{ij} P_{ij} = \langle L, vv^T \rangle$
	\end{itemize}
\end{lemma}

For brevity, we defer the proof of the sparse and low-rank corrections to Appendix. However, we provide the proof of the gluing error here to provide intuition for why these quadratic forms have constant rank rather than the a priori upper bound of $\Theta(n)$. In particular, the proof shows how the low rank arises from the recursive construction of the Silver Stepsize Schedule that creates $h^{(2n)}$ from $h^{(n)}$, modulo only changing the $n$-th and $2n$-th stepsizes (each increases the rank by $2$).

\begin{proof}[Proof of gluing error for Lemma~\ref{lem:succinct}]
	Denote by $\tilde{x}_n := x_{n-1} - b_n g_{n-1}$ and $\tilde{x}_{2n} := x_{2n-1} - b_n g_{2n-1}$ the iterates obtained by running GD with the Silver Stepsize Schedule $h^{(n)}$ from initializations $x_0$ and $x_n$, respectively. By definition of $\sig$ as a certificate for the $n$-step rate,
	\begin{align*}
		\sum_{i,j \in \{0, \dots, n-1, *\}} \sig_{ij} P_{ij} &= \tau_n \|x_0\|^2 - \|\tilde{x}_n\|^2
		\qquad \text{and } \qquad \sum_{i,j \in \{n, \dots, 2n-1, *\}} \sig_{ij} P_{ij} &= \tau_n \|x_n\|^2 - \|\tilde{x}_{2n}\|^2\
	\end{align*}
	Thus the desired quantity is equal to
	\begin{align*}
		\Big( \tau_{2n} \|x_0\|^2 - \|x_{2n}\|^2 \Big) - \sum_{i,j \in \{0,\dots,2n-1,*\}} \Theta_{ij} P_{ij}
		&= \Big( \tau_{2n} \|x_0\|^2 - \|x_{2n}\|^2 \Big) 
		- \frac{\tau_{2n}}{\tau_n}\Big( \tau_n \|x_0\|^2 - \|\tilde{x}_n\|^2 \Big) - c\Big( \tau_n \|x_n\|^2 - \|\tilde{x}_{2n}\|^2 \Big)
		\\ &= \Big( \frac{\tau_{2n}}{\tau_n} \|\tilde{x}_n \|^2 - c\tau_n \|x_n\|^2 \Big) + \Big( c\|\tilde{x}_{2n}\|^2 - \|x_{2n}\|^2 \Big)\,.
	\end{align*}
	Now by definition of GD, $x_n = x_{n-1} - a_{2n} g_{n-1}$, $x_{2n} = x_{2n-1} - b_{2n} g_{2n-1}$, $\tilde{x}_n = x_{n-1} - b_n g_{n-1}$, and $\tilde{x}_{2n} = x_{2n-1} - b_n g_{2n-1}$.  By plugging this into the above display and expanding the square, we see that the discrepancy between the $\|\tilde{x}_n\|^2$ and $\|x_n\|^2$ terms creates a quadratic form in just $x_{n-1}, g_{n-1}$, and similarly the discrepancy between the $\|\tilde{x}_{2n}\|^2$ and $\|x_{2n}\|^2$ terms creates a quadratic form in just $x_{2n-1}, g_{2n-1}$. Tracking coefficients completes the proof. 
\end{proof}

\subsection{Certificate verification}\label{ssec:rate:pf}

\begin{proof}[Proof of Theorem~\ref{thm:cert}]
	The non-negativity and $*$-sparsity properties of $\lam$ are direct from the explicit values of $\lam$; details in Appendix~\ref{app:explicit}. It therefore suffices to check the rate certificate~\eqref{eq:cert:2n}. By definition of the co-coercivity $Q_{ij}$, this certificate has two components: a linear form in $\{f_i\}_{i \in \{0,\dots,2n-1,*\}}$ and a quadratic form in $\{x_i, g_i\}_{i \in \{0, \dots, 2n-1,*\}}$. We check these two components below.

	\paragraph*{Quadratic form in iterates and gradients.} By Lemma~\ref{lem:succinct}, it suffices to show that
	\begin{align}
		E-S-L = 0\,,
	\end{align}
	where $E, S, L$ are the matrices defined in Appendix~\ref{app:succinct}.
	This amounts to checking the $10$ entries on or above the diagonal of these $4 \times 4$ matrices---elements below the diagonal need not be checked as the matrices are symmetric. By Lemma~\ref{lem:explicit}, these entries can be expressed as rational functions in $z_n, y_{2n}, z_{2n}$, which are polynomially related via~\eqref{eq:yz-defining}. Therefore, checking that these $10$ entries vanish amounts to checking that certain polynomials vanish modulo an associated ideal. This verification is rigorously automatable using standard techniques from computational algebraic geometry such as Gr\"obner bases; see e.g.~\citep{cox2013ideals,KreuzerRobbiano}. A simple script for Mathematica (or other computer algebra systems) that verifies these identities is available at the URL given in the references~\citep{MathematicaURL}. We emphasize that this is purely in the interest of brevity: verifying these identities can be done by hand, as it just amounts to straightforward (albeit tedious) algebraic cancellations.

	\paragraph*{Linear form in function values.} Recall that each $Q_{ij}$ contributes $2(M-m)(f_i - f_j)$. Thus, in order to show that all function values vanish in $\sum_{ij} \lam_{ij} Q_{ij}$, it is equivalent to show that
	\begin{align}
		\sum_{j} \lam_{ij} = \sum_j \lam_{ji} \,, \qquad \forall j \in \{0, \dots, 2n-1,*\}\,.
		\label{eq:pf-cert:netflow}
	\end{align}
	That is, the $j$-th row and column sums of $\lam$ must match, for all $j$. We call refer to these identities as \emph{netflow constraints}.
	Since $\sig$ is a valid certificate, it satisfies the netflow constraints $\sum_j \sig_{ij} = \sum_j \sig_{ji}$ for all $j \in \{0, \dots, n-1,*\}$. Thus, by construction of $\Theta$ from $\sig$, it follows that $\Theta$ satisfies the netflow constraints $\sum_j \Theta_{ij} = \sum_j \Theta_{ji}$ for all $i \in \{0, \dots, 2n-1,*\}$. Therefore, in order to prove~\eqref{eq:pf-cert:netflow}, it is equivalent to prove the netflow constraints for $\Xi + \Delta$; that is,
	\begin{align}
		\sum_j (\Xi_{ij} + \Delta_{ij}) = \sum_j (\Xi_{ji} + \Delta_{ji}), \qquad \forall i \in \{0, \dots, 2n-1,*\}\,.
		\label{eq:pf-cert:netflow-2}
	\end{align}
	The cases $i \in \{0, \dots, n-2\}$ are trivial since on these rows and columns, $\Xi$ and $\Delta$ are identically zero. The cases $i \in \{n, \dots, 2n-2\}$ are similarly trivial because on these rows and columns, $\Delta$ is identically zero and $\sum_j (\Xi_{ji} - \Xi_{ij}) = 	\Xi_{n-1,i} + \Xi_{2n-1,i} + \Xi_{*,i}
	= 0$ by construction of $\Xi$. It remains only to prove~\eqref{eq:pf-cert:netflow-2} for $i \in \{n-1,2n-1,*\}$. By the sparsity patterns of $\Xi$ and $\Delta$, this amounts to showing 
	\begin{align}
		\sum_{j \in \{n-1,2n-1,*\} \setminus \{i\}} \left(\Delta_{ij} - \Delta_{ji} \right) + \sum_{j=n}^{2n-2} \Xi_{ij} = 0, \qquad \forall i \in \{n-1,2n-1,*\}\,.
		\label{eq:pf-cert:netflow-simple}
	\end{align}
	By Lemma~\ref{lem:explicit}, these quantities can be expressed as rational functions in $z_n,y_{2n},z_{2n}$, which are polynomially related via~\eqref{eq:yz-defining}. 
	Therefore, checking that the three quantities vanish in~\eqref{eq:pf-cert:netflow-simple} amounts to checking that three polynomials vanish modulo an ideal. As mentioned above, this verification is rigorously automatable using standard computational algebra techniques; see the same URL \citep{MathematicaURL} for a simple script implementing this computation.

\end{proof}

\section{Future work}\label{sec:future}

This work removes a key stumbling block in previous analyses of optimization algorithms: we show that directly analyzing \emph{multi-step descent} can lead to improved convergence analyses. This general principle opens up a number of directions in both the design and analysis of optimization algorithms. We list a few here.

\paragraph*{Beyond GD.} Do these techniques extend to stochastic settings where gradients are noisy or only computed approximately? This is motivated by modern machine learning settings such as empirical risk minimization. What about constrained settings where projections are interleaved? Or other settings where one uses coordinate descent, proximal steps, etc.? What about second-order methods such as Newton or Interior Point methods? The modern optimization toolbox is broad, and the algorithmic opportunity of faster multi-step descent that we establish warrants re-investigating many existing algorithms that use greedy analyses.

\paragraph*{Beyond convexity.} While our techniques extend to the convex setting (see \S\ref{sssec:intro:dis:setting}), it is less clear if extensions to non-convex settings are also possible. In particular, can one prove accelerated rates for converging to an stationary point? Could this justify empirical phenomena observed in neural network training such as super-acceleration from cyclic stepsize schedules~\citep{smith2019super,smith2017cyclical}? 

\paragraph*{Faster convergence for restricted function classes.} Is faster convergence possible if the objective function is more structured? One well-motivated direction here is low-dimensional objective functions. 
It is known that faster asymptotic convergence is possible if the dimension $d$ is fixed and the number of iterations $n \to \infty$, e.g., via cutting planes. Recent work has shown that certain momentum-based modifications to GD can also surpass standard lower bounds~\citep{nem-yudin} for sufficiently large $n$~\citep{peng2023nesterov}.
Do such phenomena extend to GD with dynamic stepsizes? Altschuler's thesis~\citep[Chapter 6]{altschuler2018greed} proved that for univariate convex functions (or more generally, separable convex functions), GD achieves the fully accelerated rate $\Theta(\sqrt{\kappa} \log 1/\eps)$ via a certain (random) dynamic choices of stepsizes. Does this extend to higher dimension? What is the fundamental trade-off between $n$, $d$, and the convergence rate?

\paragraph*{Robustness.} The Silver Stepsize Schedule periodically uses extremely large step sizes, which are overly aggressive in isolation, but effective when combined with other short steps. It is natural to wonder if this dependence between iterations makes such strategies more sensitive to model misspecification, noisy gradients, inexact arithmetic, or other considerations in practical implementations. We expect this may occur, since it does for other accelerated algorithms, see e.g.,~\citep{devolder2014first}.

	\paragraph*{Acknowledgements.} JMA is grateful to his friends for their patience over the past seven years as he continually complained about how hard this problem was. 
	
	\newpage
	\appendix
	\section{Deferred details for \S\ref{sec:rate}}\label{app:rate}

Here we prove the matching lower bounds in Theorem~\ref{thm:rate}.

\paragraph*{Acceleration regime.} Suppose $n \leq n^*$. Then
	\begin{align}
		\tau_n
		= (1 - h_i)^2
		\geq
		\exp\left( - 4h_i \right)
		\geq
		\exp\left( - 8 \rho^i / \kappa \right)
		= \exp\left( - 8 n^{\log_2 \rho} / \kappa \right)
		\nonumber
		\,.
	\end{align}
	Above, the first step is by~\eqref{eq:rate:tau-h}. The third step is by the upper bound in Lemma~\ref{lem:dynamics-acceleration} and the initialization upper bound $h_0 \leq 2/\kappa$ from~\eqref{eq:pf-rate:init}. The fourth step is by definition of $i = \log_2 n$. It remains to argue the second step. This is due to the elementary inequality $1-h \geq \exp(-2h)$ which holds for $h \in (0,2/3)$ and is applicable since
	\[
	h_i \leq \rho^i h_0 \leq \rho^{i^*} h_0 \leq 2/3\,.
	\]
	Here, we used same upper bound in Lemma~\ref{lem:dynamics-acceleration}, the same initialization upper bound $h_0 \leq 2/\kappa$, and, critically, the fact that $i \leq i^*$ since we are in the acceleration regime. 
	\paragraph*{Saturation regime.} Suppose $n > n^*$. Then
	\begin{align*}
		\tau_{2n} = \left( 1 - h_{i+1}\right)^2
		\geq  \left( (1-h_i)^2 - (1-h_i)^4 \right)^2
		= \tau_n^2 (1 - \tau_n)^2
		\geq 
		\frac{\tau_n^2}{16}\,.
	\end{align*}
	where the first and third steps are by~\eqref{eq:rate:tau-h}, the second step is by the lower bound in Lemma~\ref{lem:dynamics-saturation}, and the final step is because $\tau_n \leq \tau_{n^*} \leq \exp(-1/3)$ by the rate monotonicity in Corollary~\ref{cor:rate-monotonicity} and the upper bound we proved for $\tau_{n^*}$ in the acceleration regime. By unrolling this recursion from $n$ to $n^*$, continuing to crudely bound constants for simplicity of exposition, and then plugging in the lower bound $\tau_{n^*} \geq \exp( - 8)$ from the acceleration regime, we conclude
	\begin{align}
		\tau_{n} \geq
		\left( \frac{\tau_{n^*}}{16}\right)^{n/n^*} \geq \exp\left( - \frac{n}{2n^*} \right)\,.
		\nonumber
	\end{align}

\section{Deferred details for \S\ref{sec:cert}}\label{app:cert}

Here we provide the deferred details for the proof of Theorem~\ref{thm:cert}. See \S\ref{sec:cert} for a proof overview. Three remarks on notation in this Appendix. First, after a possible translation of both $f$ and $x_0$, we assume without loss of generality that $x^* = 0$. Second, it is convenient to define the shorthand
\[
q_i := \frac{\alpha_i (1 - \tfrac{\alpha_{i+1}}{\kappa})}{\alpha_{i+1}}\,,
\]
where $\alpha_0, \dots, \alpha_{2n-1}$ index the $2n$-length Silver Stepsize Schedule $h^{(2n)}$. Third, as is standard convention, products over the empty set such as $\prod_{t=n}^{n-1} q_t$ have value $1$.

\par We begin by explicitly stating the correction components used in the recursive gluing. It is convenient to provide two equivalent versions of these expressions. In the first version (Definition~\ref{def:cert}), the low-rank correction $\Xi$ is explicitly defined for every entry, and the sparse correction $\Delta$ is typically defined as something minus the gluing component $\Theta$. Such expressions for $\Delta$ are convenient for computing the final certificate $\lam$ because $\Theta$ cancels. In the second version (Lemma~\ref{lem:explicit}), $\Xi$ is given only through its row sums (which is the only way $\Xi$ is needed for the proof of Theorem~\ref{thm:cert}), and $\Delta$ is given via explicit expressions for the subtracted entries of $\Theta$. The key benefit of this second version is that it provides explicit expressions in terms of $z_n,y_{2n},z_{2n}$ for all quantities required in the proof of Theorem~\ref{thm:cert}. For easy recall, we isolate this important fact in the following remark.

\begin{remark}[Explicit rational functions of $z_n, y_{2n}, z_{2n}$]\label{rem:simple-yz}
	While the expressions in Definition~\ref{def:cert} are somewhat involved, the key point is that all the quantities that are required in the proofs in \S\ref{sec:cert} can be expressed as rational functions in $z_n, y_{2n}, z_{2n}$. This is Lemma~\ref{lem:explicit}.
	(Note that $\tau_n, \tau_{2n}$ are by definition rational functions of $z_n, z_{2n}$, and also note that $\Xi$ is needed only through its row sums.) The upshot is that since $z_n, y_{2n}, z_{2n}$ are polynomially related by construction (see \S\ref{ssec:construction:yz}), these expressions make the rate verification a routine and rigorously automatable algebraic exercise.
\end{remark}

\begin{defin}[Corrections to the recursive gluing in Theorem~\ref{thm:cert}]\label{def:cert}
	The ``low-rank correction'' $\Xi$ is defined to be zero except that for all $j \in \{n,\dots,2n-2\}$,
	\begin{itemize}
		\item $\Xi_{n-1,j} := \phi /  \prod_{t=n}^{j-1} q_t$
		\item $\Xi_{2n-1,j} := r \phi/  \prod_{t=n}^{j-1} q_t$
		\item $\Xi_{*,j} := -(1 + r) \phi /  \prod_{t=n}^{j-1} q_t$
	\end{itemize}
	The ``sparse correction'' $\Delta$ is zero everywhere except for the following entries:
	\begin{itemize}
		\item  $\Delta_{n-1,2n-1} := 
		\phi \frac{\kap-2}{\kap(1-z_n)}$
		\item $\Delta_{2n-1,n-1} := \phi (\kap-2) \frac{y_{2n}}{1-z_n}$
		\item $\Delta_{*,n-1} := \tau_{2n}\frac{1 + \kap y_{2n}}{1 - z_n} - \frac{\tau_{2n}}{\tau_n} \sig_{*,n-1}$
		\item $\Delta_{*,2n-1} :=\frac{1+(\kap-1)z_{2n} + \kap z_{2n}^2}{(1+z_{2n})^2} -c \sig_{*,n-1}$
		\item $\Delta_{n-1,*} := - \frac{\tau_{2n}}{\tau_n} \sig_{n-1,*}$
		\item $\Delta_{2n-1,*} := \frac{2\kappa z_{2n}}{(1+z_{2n})^2} - c \sig_{n-1,*}$
	\end{itemize}
	In the above, we used as shorthand the following special values:
	\begin{itemize}
		\item $r := 1/\prod_{t=0}^{n-1} q_t$
		\item $c := \frac{\tau_{2n}}{\tau_n} \left[ r + (1 + r) \left( \frac{z_{2n} + z_n}{z_{2n} - z_n} \right) \right]$
		\item $\phi := \tau_{2n} \frac{\kappa}{\kappa - 2} \left(\frac{z_{2n} + z_n}{z_{2n} - z_n}\right)$ 
	\end{itemize}
\end{defin}

We now state the alternative expressions discussed in Remark~\ref{rem:simple-yz}.

\begin{lemma}[Alternative expressions in terms of $z_n,y_{2n},z_{2n}$]\label{lem:explicit}
	The following identities hold:
	\begin{itemize}
		\item $\Delta_{n-1,2n-1} = \frac{\tau_{2n}}{1-z_n} \left( \frac{z_{2n} + z_n}{z_{2n} - z_n}\right)$
		\item $\Delta_{2n-1,n-1} =\frac{  \kappa y_{2n }\tau_{2n}}{1-z_n}\left( \frac{z_{2n} + z_n}{z_{2n} - z_n}\right)$
		\item $\Delta_{*,n-1} = \frac{\tau_{2n}}{1-z_n}(1 + \kappa y_{2n}) - \tau_{2n}  \frac{ 1 + (\kap - 1)z_n + \kap z_n^2 }{(1-z_n)^2} $
		\item $\Delta_{*,2n-1} = \tau_{2n} \frac{1 + (\kap - 1) z_{2n} + \kap z_{2n}^2 }{(1-z_{2n})^2} - c \tau_n \frac{1 + (\kap - 1) z_n + \kap z_n^2}{(1 - z_n)^2}$
		\item $\Delta_{n-1,*} = -  \frac{2 \kappa z_n\tau_{2n}}{(1-z_n)^2}$
		\item $\Delta_{2n-1,*} = \frac{2 \kappa z_{2n} \tau_{2n}}{(1 - z_{2n})^2} - c \frac{2\kappa z_n \tau_n}{(1-z_n)^2}$
		\item$\sum_{j=n}^{2n-2} \Xi_{n-1,j} 
		= 
		\frac{1}{r} \sum_{j=n}^{2n-2} \Xi_{2n-1,j} 
		= 
		-\frac{1}{1+r}  \sum_{j=n}^{2n-2} \Xi_{*,j}
		=
		\tau_{2n} \left( \frac{\kappa z_n - 1}{1 - z_n} \right) \left( \frac{z_{2n} + z_n}{z_{2n} - z_n} \right)$ and is $0$ for $n=1,2$
		\item $r = \frac{1-z_n}{1-z_{2n}}$ 
	\end{itemize}
	\end{lemma}

	These equivalent expressions make it straightforward to prove the deferred parts of Theorem~\ref{thm:cert}.

\begin{proof}[Proof of non-negativity and $*$-sparsity in Theorem~\ref{thm:cert}]
	\underline{Checking $*$-sparsity.} Since $\sig$ satisfies $*$-sparsity, it follows immediately that $\lam_{i,*} = 0$ for all $i < 2n-1$ except possibly $i = n-1$. For this remaining case, the definition of the sparse correction $\Delta$ ensures $\lam_{n-1,*} = 0$. 
	
	\underline{Checking non-negativity.} First observe that $q_i, r, c, \phi$ are all non-negative by the bounds $z_n \leq z_{2n} \leq 1$ in Lemma~\ref{lem:yz-basic} and the bounds $ 1 \leq a_n,b_n \leq \kappa$ in Lemma~\ref{lem:ab-basic}. Next, note that all $\Theta_{ij}$ are non-negative since they are a positive multiple of some entry of $\sig$, which is non-negative by assumption of $\sig$ being a valid certificate. Thus we need only check non-negativity of $\lam_{ij} = \Theta_{ij} + \Xi_{ij} + \Delta_{ij}$ on the entries where either the correction $\Xi$ or $\Delta$ is non-zero. This non-negativity is clear from the construction for all of the sparsely-corrected entries $\lam_{ij}$ where $i \neq j \in \{n-1,2n-1,*\}$ as well as nearly all of the rank-one-corrected entries $\lam_{ij}$, namely for all $i \in \{n-1,2n-1\}$ and $j \in \{n,\dots,2n-2\}$. 
	\par It remains only to prove non-negativity of $\lam_{*,j}$ for $j \in \{n,\dots,2n-2\}$. Since $\Delta_{*,j} = 0$, this amounts to showing that $\Theta_{*,j} \geq \Xi_{*,j}$, i.e., $c \sig_{*,j-n} \geq (1+r) \phi / \prod_{t=0}^{j-1} q_t$. This follows by plugging in the explicit formulas for $\sig_{*,j}$ in Lemma~\ref{lem:explicit} and the definitions of $r$, $c$, $\phi$ in Theorem~\ref{thm:cert}.
\end{proof}	
	
	The rest of this Appendix section is organized as follows. In~\ref{app:helper-opt} and~\ref{app:helper-qi}, we provide two helper lemmas. The former explicitly computes the values of all co-coercivity multipliers to/from optimum for the $n$-step certificate $\sigma$. The latter provides useful identities involving sums and products of $q_i$. We then use these two helper lemmas in~\ref{app:explicit}  to prove the alternative expressions in Lemma~\ref{lem:explicit}, and in~\ref{app:succinct} to prove the succinct quadratic form representations in Lemma~\ref{lem:succinct}.

\subsection{Helper lemma: co-coercivities involving $x^*$}\label{app:helper-opt}

Here we compute all co-coercivity multipliers $\sig_{t,*}$ and $\sig_{*,t}$ between GD iterates $x_t$ and $x_*$.

\begin{lemma}[Co-coercivity multipliers to/from $x^*$]\label{lem:identities-opt} Consider the setup of Theorem~\ref{thm:cert}. Then:
	\begin{itemize}
		\item $\sig_{j,*} = 0$ for all $j \in \{0, \dots, n-2\}$
		\item $\sig_{n-1,*} = \frac{2\kappa z_{n}}{(1+z_{n})^2}$
		\item $\sig_{*,j} =  \frac{( \prod_{t=j}^{n-3} q_t ) (1-z_n)}{(1+z_n)^2}$ for all $j \in \{0, \dots, n-3\}$
		\item $\sig_{*,n-2} = \frac{1-z_n}{(1+z_n)^2}$
		\item $\sig_{*,n-1} = \frac{ 1 + (\kappa - 1) z_{n} + \kappa z_{n}^2 }{(1 + z_{n})^2}$ 
	\end{itemize}
\end{lemma}
\begin{proof}
	That $\sig_{j,*} = 0$ for all $j < n-1$ is trivially due to the assumption that $\sig$ satisfies the $*$-sparsity property (Definition~\ref{def:*sparsity}). The content of the lemma is solving for the other multipliers. To this end, recall that the $n$-step certificate~\eqref{eq:cert:n} establishes that the two quadratic forms $\tau_n \|x_0\|^2 - \|x_n\|^2$ and $\sum_{i\neq j \in \{0, \dots, n-1, *\}} \sig_{ij} Q_{ij}$ are equal modulo the ideal generated by the equations $x_{t+1} = x_t - \alpha_t g_t$ for all $t \in \{0, \dots, n-1\}$. Thus if we expand both these quadratic forms by replacing, for every $t > 0$, the iterate $x_t$ with $x_0 - \sum_{s=0}^{t-1} \alpha_s g_s$, then the resulting two quadratic forms (now in the variables $\{x_0, g_0, \dots, g_{n-1}\}$) are equal, and in particular the coefficient of any term must match. 
	\par We prove this lemma by solving the equations that come from matching the coefficients for the terms $\|x_0\|^2$ and $\langle x_0, g_i \rangle$ for $i \in \{0, \dots, n-1\}$. By expanding the definition of the co-coercivity, it is evident that only the co-coercivities of the form $Q_{t,*}$ and $Q_{*,t}$ contributes coefficients for these terms. In particular, matching the coefficients for the term $\|x_0\|^2$ gives the equation
	\begin{align}
		\tau_n - 1 = -\frac{1}{\kappa} \left(  \sig_{n-1,*} + \sum_{j=0}^{n-1} \sig_{*,j} \right)\,,
		\label{eq:pf-opt:x0}
	\end{align}
	matching the coefficients for the term $\langle x_0, g_{n-1} \rangle$ gives the equation
	\begin{align}
		\alpha_{n-1}
		=
		\frac{1}{\kappa} \sig_{n-1,*} + \sig_{*,n-1}\,,
		\label{eq:pf-opt:x0gn-1}
	\end{align}
	and matching the coefficients for the other terms $\langle x_0, g_t \rangle$ gives the equations
	\begin{align}
		\alpha_t
		=
		\frac{\alpha_t}{\kap} \left( \sig_{n-1,*} + \sum_{i=t+1}^{n-1} \sig_{*,i} \right) + \sig_{*,t}
		\,, \qquad \forall t \in \{0, \dots, n-2\}
		\label{eq:pf-opt:x0gt}
	\end{align}
	Intuitively, these equations can be back-solved since they are (essentially) already in triangular form. Below we detail a simple way to do this by hand.
	\par First, we obtain the claimed expression for $\sig_{n-1,*}$ by combining the netflow equation $\sig_{n-1,*} = \sum_{j=0}^{n-1} \sig_{*,j}$ with~\eqref{eq:pf-opt:x0}, re-arranging, and plugging in the definition of $\tau_n = (\frac{1-z_n}{1+z_n})^2$.
	\par Next, we obtain the claimed expression for $\sig_{*,n-1}$ by plugging the now-proved value of $\sig_{n-1,*}$ into~\eqref{eq:pf-opt:x0gn-1} and using the fact that the Silver Stepsize $\alpha_{n-1} = b_n = (1 + \kap z_n) / (1+z_n)$, see \S\ref{sec:construction}.
	\par Next, we obtain the claimed expression for $\sig_{*,n-2}$ by plugging the now-proved values for $\sig_{*,n-1}$ and $\sig_{n-1,*}$ into~\eqref{eq:pf-opt:x0gt}, for $t=n-2$, and using the fact that Silver Stepsize $\alpha_{n-2} = \kappa/(\kappa -1)$, see \S\ref{sec:construction}.
	\par To solve for the remaining variables $\{\sig_{*,j}\}_{j \in \{0, \dots, n-3\}}$, we could continue back-solving by plugging into~\eqref{eq:pf-opt:x0gt}. However, there is a simpler approach: by subtracting the $j$-th equation~\eqref{eq:pf-opt:x0gt} from the $(j+1)$-th equation~\eqref{eq:pf-opt:x0gt}, the partial sums telescope. After re-arranging, this gives the recurrence
	\begin{align}
		\sig_{*,j} = q_{j} \sig_{*,j+1}\,, \qquad \forall j \in \{0, \dots, n-3\}\,.
	\end{align}
	By plugging in the now-proved value of $\sig_{*,n-2}$ as the base case for this backwards recurrence, we obtain the claimed expression for $\sig_{*,j}$ for all $j \in \{0, \dots, n-3\}$. 
\end{proof}

\subsection{Helper lemma: identities involving $q_i$}\label{app:helper-qi}

Here we provide useful identities involving $q_i$. This enables expressing $r = \prod_{t=0}^{n-1} 1/q_t$ and sums of the form $\sum_{j=n}^{2n-2} \prod_{t=n}^{j-1} 1/q_t$ in terms of the Normalized Silver Stepsizes $z_n$ and $z_{2n}$. We will use this to prove the final two items of Lemma~\ref{lem:explicit} in Appendix~\ref{app:explicit}. Note that for simplicity, we state these identities for $n \geq 4$ since the low-rank component $\Xi$ (which is what this lemma is used to compute) is identically zero for small $n$.

\begin{lemma}[Identities involving products of $q_i$]\label{lem:identities-qi}
    For $n \geq 4$:
	\begin{itemize}
		\item $	\prod_{t=0}^{n-3} \frac{1}{q_t} = \frac{\kappa - 2}{\kappa (1 - z_n)}$
		\item $\prod_{t=0}^{n-1} \frac{1}{q_t}
		= \frac{1-z_n}{1-z_{2n}} = r$ 
		\item $\sum_{j=n}^{2n-2} \prod_{t=n}^{j-1} \frac{1}{q_t} =
		\frac{(\kappa - 2)(\kappa z_n - 1)}{\kappa(1-z_n)} 
		$
	\end{itemize}
\end{lemma}

The proof exploits the following elementary identity. We remark that this identity has a probabilistic interpretation, although we do not make explicit use of it. This interpretation is for the case $c_t \in [0,1]$; then $c_t$ can be viewed as the probability that the $t$-th coin is heads, hence $\sum_{t=0}^T c_t \prod_{i=t+1}^T ( 1 - c_i) = 1 - \prod_{t=0}^T (1 - c_t) $ are two expressions for the probability that at least one coin is heads. Below, recall the standard convention that the product over an empty set is $1$.

\begin{lemma}\label{lem:disjunction}
	For any non-negative integer $T$ and any real numbers $c_0, \dots, c_T$,
	\begin{align*}
	 \sum_{t=0}^T c_t \prod_{i=t+1}^T ( 1 - c_i) = 1 - \prod_{t=0}^T (1 - c_t)\,.
	\end{align*}
\end{lemma}
\begin{proof}
	We prove by induction. The base case $T=0$ is trivial. For the inductive step, assume true for $T$; then the claim holds for $T+1$ because 
	$\sum_{t=0}^{T+1} c_t \prod_{i=t+1}^{T+1} ( 1 - c_i)
	=
	c_{T+1} + (1 - c_{T+1}) ( \sum_{t=0}^T c_t \prod_{i=t+1}^T (1 - c_i))
	=
	c_{T+1} + (1 - c_{T+1}) ( 1 - \prod_{t=0}^T (1 - c_t))
	= 1 - \prod_{t=0}^{T+1} (1 - c_t)$.
\end{proof}

\begin{proof}[Proof of Lemma~\ref{lem:identities-qi}]
	Since $\sig$ is a valid certificate, it satisfies the netflow constraint $\sum_j \sig_{*,j} = \sum_j \sig_{j,*}$. By using the values for these entries (Lemma~\ref{lem:identities-opt}) and re-arranging, we obtain
	\begin{align}
		\sum_{j=0}^{n-2} \prod_{t=j}^{n-3} q_j = \kappa z_n - 1\,.
		\label{eq:pf-lem-qi-1}
	\end{align}
	Next, we compute this quantity in a different way. By definition of $q_i = \frac{\alpha_i (1 - \alpha_{i+1}/\kappa)}{\alpha_{i+1}}$,
	multiplying consecutive $q_i$ yields a telescoping product, namely
	\begin{align}
		\prod_{t=j}^{n-3}  q_t = \frac{\alpha_j}{\alpha_{n-2}} \prod_{t=j+1}^{n-2} \left( 1 - \frac{\alpha_t}{\kappa} \right)\,.
		\label{eq:pf-lem-qi-2}
	\end{align}
	In particular, this implies that
	\begin{align}
		\sum_{j=0}^{n-2} \prod_{t=j}^{n-3} q_j 
		= 
		\frac{\kappa}{\alpha_{n-2}} \,\Bigg( \sum_{j=0}^{n-2} \frac{\alpha_j}{\kappa} \prod_{t=j+1}^{n-2} \left(1 - \frac{\alpha_t}{\kappa} \right) \Bigg)
		=
		(\kappa - 1)\, \Bigg( 1 - \prod_{t=0}^{n-2} \left(1 - \frac{\alpha_t}{\kappa} \right) \Bigg)\,.
		\label{eq:pf-lem-qi-3}
	\end{align}
	Above, the second step uses Lemma~\ref{lem:disjunction} and the fact that $\alpha_{n-2} = a_2 = \kappa/(\kappa - 1)$ by construction of the Silver Stepsize Schedule (see \S\ref{sec:construction}). Now by combining~\eqref{eq:pf-lem-qi-1} and~\eqref{eq:pf-lem-qi-3}, we obtain
	\begin{align*}
		\prod_{t=0}^{n-2} \left( 1 - \frac{\alpha_t}{\kap} \right) = \frac{\kappa\,(1-z_n)}{\kappa - 1}\,.
	\end{align*}
	By using again the telescoping property~\eqref{eq:pf-lem-qi-2} and the fact that $\alpha_0 = \alpha_{n-2}$ by construction of the Silver Stepsize Schedule (see \S\ref{sec:construction}), we conclude that
	\begin{align}
		\prod_{t=0}^{n-3} q_t
		=
		\frac{\alpha_0}{\alpha_{n-2}} \prod_{t=1}^{n-2} \left( 1 - \frac{\alpha_t}{\kap} \right)
		=
		\frac{ \prod_{t=0}^{n-2} (1 - \alpha_t/\kap)}{1 - \alpha_0/\kap}
		=
		\frac{\kappa\,(1 - z_n)}{\kappa - 2}\,.
		\label{eq:pf-lem-qi-4}
	\end{align}
	This proves the first claim. 
	\par For the second claim, observe that
	\begin{align*}
		q_{n-1} q_{n-2} 
		= \left( 1 - \frac{\alpha_n}{\kap} \right) \left( 1 - \frac{\alpha_{n-1}}{\kap} \right) 
		= \frac{\kap - 2}{\kap - 1} \cdot \frac{\kap-1}{\kap(1+y_{2n})}
		= \frac{\kap - 2}{\kappa (1 + y_{2n})}\,,
	\end{align*}
	where above we have simplified by using the facts that $\alpha_{n-2} = \alpha_n = a_2 = \kappa/(\kappa - 1)$ and $\alpha_{n-1} = a_{2n}$ by construction of the Silver Stepsize Schedule (see \S\ref{sec:construction}), as well as the re-parameterization of $y_{2n}$ in terms of $a_{2n} = \alpha_{n-1}$. Multiplying the above two displays yields
	\begin{align*}
		\prod_{t=0}^{n-1} q_t = \frac{1-z_n}{1+y_{2n}}\,.
	\end{align*}
	The proof of the second claim is then complete by using the identity $(1 - z_n)^2 = (1-z_{2n})(1+y_{2n})$ which follows from the recurrence construction of the $y_n,z_n$ sequences in \S\ref{sec:construction}.
	\par Finally, for the third claim, divide~\eqref{eq:pf-lem-qi-1} by~\eqref{eq:pf-lem-qi-4} to conclude the desired identity
	\begin{align*}
		\sum_{j=0}^{n-2} \prod_{t=0}^{j-1} \frac{1}{q_t}
		=
		\frac{\sum_{j=0}^{n-2} \prod_{t=j}^{n-3} q_j}{\prod_{t=0}^{n-3} q_t} 
		=
		\frac{(\kappa - 2)(\kappa z_n - 1)}{\kappa (1 - z_n)}\,.
	\end{align*}
\end{proof}

\subsection{Recursive gluing as a rational function of $z_n, y_{2n}, z_{2n}$}\label{app:explicit}

\begin{proof}[Proof of Lemma~\ref{lem:explicit}]
	The expressions for $\Delta$ are immediate from Definition~\ref{def:cert} and Lemma~\ref{lem:identities-opt}. The expressions for the row sums of $\Xi$ are immediate by definition of $\Xi$ and the final identity in Lemma~\ref{lem:identities-qi}. The expression for $r$ follows from Lemma~\ref{lem:identities-qi}.
\end{proof}

\subsection{Recursive gluing as a succinct quadratic form}\label{app:succinct}

Here we provide details for Lemma~\ref{lem:succinct} and its proof. The definitions of the coefficient matrices $E$, $S$, and $L$ are as follows. Note that these matrices are symmetric, thus for shorthand we simply write a tilde for their lower-triangular elements. 
\begin{align}
	\small
	E := \begin{bmatrix}
		\frac{\tau_{2n}}{\tau_n} - c\tau_n & c \tau_n a_{2n}  - \frac{\tau_{2n}}{\tau_n} b_n & 0 & 0 \\
		\tilde & \frac{\tau_{2n}}{\tau_n} b_n^2  - c\tau_n a_{2n}^2 & 0 & 0 \\
		\tilde & \tilde & c-1 & b_{2n} - c b_n  \\
		\tilde & \tilde & \tilde & c b_n^2 - b_{2n}^2 \\				
	\end{bmatrix}
\nonumber
\end{align}
and
\begin{align}
	\small
	S := \begin{bmatrix}
		-\frac{1}{\kappa}\sum\limits_{t \neq n-1} (\Delta_{n-1,t} + \Delta_{t,n-1}) 
		& \sum\limits_{t \neq n-1} (\frac{\Delta_{n-1,t}}{\kappa} + \Delta_{t,n-1}) 
		& \frac{1}{\kappa} (\Delta_{n-1,2n-1} + \Delta_{2n-1,n-1}) 
		& -\Delta_{n-1,2n-1}- \frac{\Delta_{2n-1,n-1}}{\kappa}  
		\\				
		\tilde 
		& - \sum\limits_{t \neq n-1} (\Delta_{n-1,t} + \Delta_{t,n-1}) 
		& - \Delta_{2n-1,n-1} - \frac{\Delta_{n-1,2n-1}}{\kappa}
		& \Delta_{n-1,2n-1} + \Delta_{2n-1,n-1}
		\\				
		\tilde 
		& \tilde 
		& -\frac{1}{\kappa} \sum\limits_{t \neq 2n-1} ( \Delta_{2n-1,t} + \Delta_{t,2n-1}) 
		& \sum\limits_{t \neq 2n-1} (\frac{\Delta_{2n-1,t}}{\kappa} + \Delta_{t,2n-1}) 
		\\				
		\tilde 
		& \tilde 
		& \tilde 
		& - \sum\limits_{t \neq 2n-1} (\Delta_{2n-1,t} + \Delta_{t,2n-1}) 			
	\end{bmatrix}
\nonumber
\end{align}
and $L :=  \phi \left( L^{(n-1)} + r L^{(2n-1)} \right)$, where
	\begin{align}
		\small
	L^{(n-1)} &:= 
	\frac{\kap-2}{\kap (1 - z_n)}
	\begin{bmatrix}
		z_n - 2 + \tfrac{1}{\kappa} & a_{2n} (1 - z_n) + \tfrac{\kappa - 1}{\kappa} & 1 - \tfrac{1}{\kappa} & 0 \\
		\tilde &  2a_{2n} (z_n - 1) - (\kappa z_n - 1) &  -1 + \tfrac{1}{\kappa}  & 0 \\    		
		\tilde & \tilde & 0 & 0 \\
		\tilde & \tilde & \tilde & 0 \\
	\end{bmatrix}
	\nonumber 
	\\
	L^{(2n-1)} &:= 
	\frac{\kap-2}{\kap (1 - z_n)}
	\begin{bmatrix}
		0
		&
		0
		&
		-1+z_n
		&
		1-z_n
		\\
		\tilde
		&
		0
		&
		a_{2n} (1-z_n)
		&
		-a_{2n} (1-z_n)
		\\
		\tilde
		&
		\tilde
		&
		2 - z_n - \tfrac{1}{\kappa}
		&
		-1+z_n
		\\
		\tilde
		&
		\tilde
		&
		\tilde
		&
		1 - \kap z_n
	\end{bmatrix}
	\nonumber
\end{align}
and is zero for $n=1,2$.

\begin{proof}[Proof of remaining parts of Lemma~\ref{lem:succinct}]
		The identity for the sparse correction is immediate by plugging in the definition of $P_{ij}$, simplifying $x^* = g^* = 0$, and collecting terms. It remains to show the identity for the low-rank correction. Suppose $n \geq 4$, else $\Xi$ is identically zero and the claim is trivial. We claim that
		\begin{align}
			\sum_{j=n}^{2n-2} \prod_{t=n}^{j-1} \frac{1}{q_j} \left( P_{n-1,j} - P_{*,j} \right) 
			&= \langle L^{(n-1)}, vv^T \rangle 
			\label{eq:lem-lr:n-1}
			\\ 
			\sum_{j=n}^{2n-2} \prod_{t=n}^{j-1} \frac{1}{q_j} \left( P_{2n-1,j} - P_{*,j} \right) 
			&= \langle L^{(2n-1)}, vv^T \rangle
			\,.
			\label{eq:lem-lr:2n-1}
		\end{align}
		These identities suffice because by the definition of $\Xi$ and $L$, we then have
		\[
			\sum_{ij} \Xi_{ij}P_{ij}
			=
			\phi \sum_{j=n}^{2n-2} \prod_{t=n}^{j-1} \frac{1}{q_j} \left( \left( P_{n-1,j} - P_{*,j} \right) - r \left( P_{2n-1,j} - P_{*,j} \right) \right)
			=
			\langle L, vv^T \rangle\,.
		\]

	\par We prove~\eqref{eq:lem-lr:n-1}; the proof of~\eqref{eq:lem-lr:2n-1} is entirely analogous and thus omitted for brevity. By expanding the squares in the definition of $P_{ij}$, simplifying $x^* = g^* = 0$, and collecting terms, 
	\begin{align*}
		P_{n-1,j} - P_{*,j}
		=
		- \frac{1}{\kappa} \|x_{n-1}\|^2 - \|g_{n-1}\|^2 + \frac{2}{\kappa} \langle x_{n-1}, g_{n-1} \rangle + 2 \langle x_{n-1} - g_{n-1} , \frac{x_j}{\kappa} - g_j\rangle\,.
	\end{align*}
	By expanding $x_{j} = x_{n-1} - \sum_{i=n-1}^{j-1} \alpha_i g_i$ using the definition of GD, and then collecting terms,
	\begin{align*}
		P_{n-1,j} - P_{*,j}
		=
		\frac{1}{\kappa} \|x_{n-1}\|^2 + \left( \frac{2\alpha_{n-1}}{\kappa}  - 1\right) \|g_{n-1}\|^2 - \frac{2\alpha_{n-1}}{\kappa} \langle x_{n-1}, g_{n-1} \rangle
		- 2 \langle x_{n-1} - g_{n-1} , g_j + \frac{1}{\kappa} \sum_{i=n}^{j-1} \alpha_i g_i\rangle \,.
	\end{align*}
	Since the first three of these four summands are independent of $j$, we conclude
	\begin{align}
		\sum_{j=n}^{2n-2} \prod_{t=n}^{j-1} \frac{1}{q_j} \left( P_{n-1,j} - P_{*,j} \right) 
		&=
		\left( \sum_{j=n}^{2n-2} \prod_{t=n}^{j-1} \frac{1}{q_j} \right)  \Bigg( \frac{1}{\kappa} \|x_{n-1}\|^2 + \left( \frac{2\alpha_{n-1}}{\kappa}  - 1\right) \|g_{n-1}\|^2 - \frac{2\alpha_{n-1}}{\kappa} \langle x_{n-1}, g_{n-1} \rangle \Bigg) \nonumber
		\\ &- 2 \left\langle x_{n-1} - g_{n-1} , \sum_{j=n}^{2n-2} \prod_{t=n}^{j-1} \frac{1}{q_j} \Big(  g_j + \frac{1} {\kappa} \sum_{i=n}^{j-1} \alpha_i g_i \Big) \right\rangle \,.
		\label{eq:pf-lr:1}
	\end{align}
	For the first term, use Lemma~\ref{lem:identities-qi} and the fact that $q_j = q_{j-n}$ for $j \in \{n, \dots, 2n-2\}$ to obtain
	\begin{align}
		\sum_{j=n}^{2n-2} \prod_{t=n}^{j-1} \frac{1}{q_j} 
		=
		\sum_{j=0}^{n-2} \prod_{t=0}^{j-1} \frac{1}{q_j} 
		=
		\frac{(\kappa - 2)(\kappa z_n-1)}{\kappa (1 - z_n)}\,.
		\label{eq:pf-lr:2}
	\end{align}
	For the second term, observe that
	\begin{align}
		\sum_{j=n}^{2n-2} \prod_{t=n}^{j-1} \frac{1}{q_j} \left(  g_j + \frac{1} {\kappa} \sum_{i=n}^{j-1} \alpha_i g_i \right)
		&= 
		\frac{1}{\alpha_n} \sum_{j=n}^{2n-2} \alpha_j \prod_{t=n+1}^{j} \frac{1}{(1 - \alpha_t/\kappa)} \left(  g_j + \frac{1}{\kappa} \sum_{i=n}^{j-1} \alpha_i g_i \right)
		\nonumber \\ &= 
		\frac{1}{\alpha_n} \sum_{i=n}^{2n-2} \alpha_i g_i \, \Bigg[
		\frac{1}{\prod_{t=n+1}^i (1 - \alpha_t/\kappa)}
		+
		\sum_{j=i+1}^{2n-2} \frac{\alpha_j}{\kappa \prod_{t=n+1}^j (1 - \alpha_t/\kappa)}
		\Bigg]
		\nonumber \\ &= 
		\frac{1}{\alpha_n \prod_{t=n+1}^{2n-2}(1 - \alpha_t/\kappa)} \sum_{i=n}^{2n-2} \alpha_i g_i \, \Bigg[
		\prod_{t=i+1}^{2n-2}(1 - \alpha_t/\kappa) + \sum_{j=i+1}^{2n-2} \frac{\alpha_j}{\kappa} \prod_{t=j+1}^{2n-2} (1 - \alpha_t/\kappa)
		\Bigg]
		\nonumber \\ &= \frac{1}{\alpha_n \prod_{t=n+1}^{2n-2}(1 - \alpha_t/\kappa)} \sum_{i=n}^{2n-2} \alpha_i g_i 
		\nonumber \\ &= \frac{1}{\alpha_n \prod_{t=n+1}^{2n-2}(1 - \alpha_t/\kappa)} \Bigg[ x_{n-1} - x_{2n-1} - \alpha_{n-1} g_{n-1} \Bigg]
		\nonumber \\&=
		\frac{(\kappa-1)(\kappa-2)}{\kappa^2(1 - z_n)} \Bigg[ x_{n-1} - x_{2n-1} - \alpha_{n-1} g_{n-1} \Bigg]\,.\label{eq:pf-lr:3} 
	\end{align}
	Above, the first step is by definition of $q_j$ and telescoping. The second step is by re-arranging sums. The third step is by factoring out the product. The fourth step is because $\sum_{j=i+1}^{2n-2} \tfrac{\alpha_j}{\kappa} \prod_{t=j+1}^{2n-2} (1 - \tfrac{\alpha_t}{\kappa}) = 1 - \prod_{t=i+1}^{2n-2} ( 1 - \tfrac{\alpha_t}{\kappa})$ by Lemma~\ref{lem:disjunction}. The fifth step is by definition of GD. The final step uses $\alpha_n \prod_{t=n+1}^{2n-2}(1 - \tfrac{\alpha_t}{\kappa}) = \alpha_{2n-2} \prod_{t=n}^{2n-3} q_t$, Lemma~\ref{lem:identities-qi}, and the fact that $\alpha_{2n-2} = a_2 =  \kappa/(\kappa-1)$.
	\par By combining~\eqref{eq:pf-lr:1},~\eqref{eq:pf-lr:2}, and~\eqref{eq:pf-lr:3}, we conclude that
	 the desired quantity is equal to
	\begin{align*}
		\sum_{j=n}^{2n-2} \prod_{t=n}^{j-1} \frac{1}{q_j} \left( P_{n-1,j} - P_{*,j} \right) 
		= 
		\frac{\kappa - 2}{\kappa^2(1-z_n)} \Bigg[
		&(\kappa z_n - 1) \Big( \|x_{n-1}\|^2 - 2\alpha_{n-1} \langle x_{n-1}, g_{n-1} \rangle + (2\alpha_{n-1} - \kappa) \|g_{n-1}\|^2 \Big)
		\nonumber \\&- 2(\kappa-1) \langle x_{n-1} - g_{n-1}, x_{n-1} - x_{2n-1} - \alpha_{n-1} g_{n-1} \rangle
		\Bigg]\,.
	\end{align*}
	Expanding the inner product and canceling terms completes the proof of~\eqref{eq:lem-lr:n-1}.
\end{proof}

	\footnotesize
	\addcontentsline{toc}{section}{References}
	\bibliographystyle{plainnat}
	\bibliography{hedging}{}

\end{document}